\documentclass[12pt]{amsart}
\usepackage[english]{babel}
\usepackage[utf8]{inputenc}
\usepackage{lipsum}
\usepackage{mathrsfs}
\usepackage{stix}
\usepackage{fullpage}
\usepackage{mathtools}
\usepackage{amsmath}
\usepackage{tikz-cd}
\newtheorem{defi}{Definition}[section]

\newtheorem{lema}[defi]{Lemma}
\newtheorem{teo}[defi]{Theorem}
\newtheorem{rem}[defi]{Remark}

\newtheorem{coro}[defi]{Corollary}
\newtheorem{pro}[defi]{Proposition}
\newtheorem*{rem*}{Remark}

\newcommand\mysim{\stackrel{\mathclap{\normalfont\mbox{\tiny{$m_N , \rho , \delta$}}}}{\sim}}

\newcommand{\ha}{{\mathscr H}}

\newcommand{\hb}{\mathcal{H}}

\newcommand{\interior}[1]{%
  {\kern0pt#1}^{\mathrm{o}}%
}

\newcommand{\I}{\mathcal{I}}
\newcommand{\T}{\mathbb{T}}
\newcommand{\Td}{\mathbb{T}^d}
\newcommand{\C}{\mathbb{C}}
\newcommand{\Q}{\mathbb{Q}}

\newcommand{\R}{\mathbb{R}}
\newcommand{\N}{\mathbb{N}}
\newcommand{\Z}{\mathbb{Z}}

\newcommand{\esp}{\text{  }}

\usepackage[colorlinks=false]{hyperref}
\renewcommand\eqref[1]{(\ref{#1})} 

\begin{document}

\title[H{\"O}RMANDER CLASSES OF PSEUDO DIFFERENTIAL OPERATORS OVER THE COMPACT GROUP OF $p$-ADIC INTEGERS]
 {H{\"O}RMANDER CLASSES OF PSEUDO DIFFERENTIAL OPERATORS OVER THE COMPACT GROUP OF $p$-ADIC INTEGERS}

\author{
  J.P. Velasquez-Rodriguez
}

\newcommand{\Addresses}{{
  \footnotesize
  J.P.~VELASQUEZ-RODRIGUEZ, \textsc{Department of Mathematics: Analysis, Logic and Discrete Mathematics, Ghent University, Belgium.}\par\nopagebreak
  \textit{E-mail address:} \texttt{JuanPablo.VelasquezRodriguez@UGent.be}
}
}
\thanks{The 
	author was supported by the FWO Odysseus 1 grant G.0H94.18N: Analysis and Partial Differential Equations.}

\subjclass[2010]{Primary; 43A75, 47G30; Secondary: 47B38, 35S05. }

\keywords{Infinite matrices, Fourier Analysis, Pseudo-differential Operators, p-Adic integers, H{\"o}rmander classes, Hypoelliptic operators}

\date{\today}
\begin{abstract}
The purpose of this paper is to introduce new definitions of H{\"o}rmander classes for pseudo-differential operators over the compact group of $p$-adic integers. Our definitions possesses a symbolic calculus, asymptotic expansions and parametrices, together with an interesting relation with the infinite matrices algebras studied by K. Gr{\"o}chenig and S. Jaffard. Also, we show how our definition of H{\"o}rmander classes is related to the definition given in the toroidal case by M. Ruzhansky and V. Turunen. In order to show the special properties of our definition, in a later work, we will study several spectral properties in terms of the symbol for pseudo-differential operators in the H{\"o}rmander classes here defined.
\end{abstract}
\maketitle
\tableofcontents
\section{\textbf{Introduction}}
In this note we study pseudo-differential operators over the compact group $\Z_p$ of $p$-adic integers making an analysis analogous to the Ruzhansky-Turunen theory summarised in \cite{ruzhansky1}. 

The theory of pseudo-differential operators over non-archimedean spaces have been widely studied \cite{padic1, zuniga2, zuniga1}, and it currently is a very active brand of mathematics motivated mostly for its connections with mathematical physics \cite{fisicapadica, vladimirov} and other disciplines like medicine \cite{padic2}, biology or social sciences \cite{p-adicgod}. One of the most important pseudo-differential equations is the $p$-adic analogue of the heat equation $\frac{\partial f}{\partial t} + D^s f =0$, where $D^s$ is the Vladimirov operator, the $p$-adic counterpart of the Laplacian. This is the master equation that describes a Markov process without second kind discontinuities and, as it is discussed by Kochubei \cite{p-adicball1, p-adicball2}, it governs several complex systems. 

The Vladimirov operator is initially defined for some authors trough the Fourier transform as a pseudo-differential operator, but one can find a more explicit formula for it, as the reader may find in the reference book \cite{Dyadic}. In the same way as the classical theory of pseudo-differential operators, which appears as a generalization of the theory of partial differential operators, the theory in the $p$-adic case appears after a definition of a logical derivative as an operator allowing to estimate the rate of change of functions in the context of Walsh-Fourier analysis. For example when $p=2$ there exist a definition of \emph{dyadic derivative} due to J. E. Gibbs. Gibbs definition allows one to measure the rate of change of a function in the context of the Walsh dyadic analysis, and it has several important applications, as the reader may find in \cite{Dyadic}. For general $p$, a concept of $p$-adic derivative is defined by C.W. Onneweer in \cite{Onneweer1977} among others, see \cite{Dyadic} for a complete list, extending Gibbs definition to a $p$-adic or $p$-series field. One can see that, via Fourier transform, the Vladimirov operator and the Gibbs logical derivative coincide (except, maybe, for a constant) because they define (almost) the same Fourier multiplier, so we have a whole theory of equations in logical derivatives as motivation for the theory of pseudo-differential operators in $\Z_p$. 

Despite the intense research on ultrametric pseudo-differential equations, the theory of non-archimedean pseudo-differential operators, from the point of view of H{\"o}rmander classes, is not so well developed compared with the classical version of the theory on real and complex manifolds. Moreover, the connections between both theories remain unexplored, and most of the methods in the classical theory have no analogy in the $p$-adic case. The purpose of this paper is to apply the ideas of the Ruzhansky-Turunen theory to the compact group of $p$-adic integes $\Z_p$ since in this case the analysis of pseudo-differential operators can be done in a very similar way to the toroidal case, even though it seems like this fact has not been exploited sistematically. What we want to show in this work is that in this setting there exists a notion of symbol classes defined by the usual Hörmander conditions, and the corresponding operator classes have similar properties to the toroidal classes $S^m_{0,0} (\T^d \times \Z^d)$, see \cite{ruzhansky1} for the definition. Moreover we provide a pseudo-differential calculus in $\Z_p$ with parametrices for elliptic operators. 

It is worth to emphasise two important facts. The first is that, since connected abelian groups are isomorphic to $\T^d \times  \R^{n-d}$, for appropriated $n$ and $d$, the study of operators in the group is reduced to the study of operator acting on $\T^d$ and $\R^d$. The case of abelian non-connected groups, such as any profinite groups, is different, but the analysis on $\Z_p$ will serve us to give some light about it, especially for compact Vilenkin groups as we will discuss in a later work. The second is that, even when we are talking about $\Z_p$, a subset of a non-archimedean space, we will barely use the non-archimedean structure during the development od our work. It is because many of the properties here studied rely mostly on the Hilbert space structure of $L^2 (\Z_p)$ and the known Fourier analysis on compact abelian groups. Anyway in the last three sections we will explain the implications of the non-archimidean structure and we will use them in a later work to study important spectral properties of pseudo-differential operators \cite{spectraltheoryp-adicpseudos}.  

The most important applications of the theory that we intend to develop here are, first, to prove the regularity of solutions to certain pseudo-differential equations, and second, to put in terms of the symbol important properties such as $L^r$-boundedness, compactness, belonging to Schatten classes and nuclearity, Riesz spectral theory, Fredholmness, ellipticity and Gohber's lemma, among others. We refer the reader to the work \cite{spectraltheoryp-adicpseudos} where we study several spectral properties of pseudo-differential operators in the H{\"o}rmander classes here defined. 

\section{\textbf{Preliminaries}}
Along this article $p$ will denote a fixed prime number. The field of $p$-adic numbers $\Q_p$ is defined as the complection of the field of rational numbers $\Q$ with respect to the $p$-adic norm $|\cdot|_p$ which is defined as \[|x|_p := \begin{cases}
0 & \esp \text{if} \esp x=0, \\ p^{-\gamma} & \esp \text{if} \esp x= p^{\gamma} \frac{a}{b},
\end{cases}\]where $a$ and $b$ are integers coprime with $p$. The integer $\gamma:= ord(x)$, with $ord(0) := + \infty$, is called the $p$-adic order of $x$. The unit ball of $\Q_p$ with the $p$-adic norm is called the compact group of $p$-adic integers and it will be denoted by $\Z_p$. Any $p$-adic number $x \neq 0$ has a unique expansion of the form $$x = p^{ord(x)} \sum_{j=0}^{\infty} x_j p^j,$$where $x_j \in \{0,1,...,p-1\}$ and $x_0 \neq 0$. By using this expansion, we define the fractional part of $x \in \Q_p$, denoted by $\{x\}_p$, as the rational number\[\{x\}_p := \begin{cases}
0 & \esp \text{if} \esp x=0 \esp \text{or} \esp ord(x) \geq 0, \\ p^{ord(x)} \sum_{j=0}^{-ord(x)-1} x_j p^j,& \esp \text{if} \esp ord(x) <0.
\end{cases}\] It is known that $\Z_p$ is a compact totally disconected abelian group. Its dual group in the sense of Pontryagin, the collection of characters of $\Z_p$, will be denoted by $\widehat{\Z}_p$. The dual group of the $p$-adic integers is known to be the Pr{\"u}fer group $\Z (p^{\infty})$,  the unique $p$-group in which every element has $p$ different $p$-th roots. The Pr{\"u}fer group may be identified with the quotient group $\Q_p/\Z_p$. In this way the characteres of the group $\Z_p$ may be written as $$\chi_p (\xi  x) := e^{2 \pi i \{x \xi \}_p}, \esp \esp x \in \Z_p, \esp \xi \in \widehat{\Z}_p 
\cong \Q_p / \Z_p .$$

By the Peter-Weyl theorem the elements of $\widehat{\Z}_p$ constitute an orthonormal basis for the Hilbert space $L^2 (\Z_p)$, which provide us a Fourier analysis for suitable functions defined on $\Z_p$ in such a way that the formula $$f(x) = \sum_{\xi \in \widehat{\Z}_p} \widehat{f}(\xi) \chi_p (\xi  x),$$holds almost everywhere in $\Z_p$. Here $\widehat{f}$ denotes the Fourier transform of $f$ in turn defined as $$\widehat{f}(\xi):= \int_{\Z_p} f(x) \overline{\chi_p (\xi  x)}dx,$$where $dx$ is the normalised Haar measure in $\Z_p$. The above series are called the Fourier series of the function $f$.

Using the representation of functions in its Fourier series, for a given densely defined linear operator $$T: \hb^\infty :=Span\{\chi_p (\xi x)\}_{\xi \in \widehat{\Z}_p } \subset D(T) \subseteq L^2 (\Z_p) \to L^2 (\Z_p),$$we can define its associated symbol $\sigma_T (x, \xi)$ by formula $$\sigma_T (x, \xi) = \overline{\chi_p (\xi  x)} T \chi_p (\xi  x).$$In this way we can think on any densely defined linear operator as a linear operator given in terms of its symbol by formula $$Tf(x) = \sum_{\xi \in  \widehat{\Z}_p} \sigma_T (x,\xi) \widehat{f}(\xi) \chi_p (\xi  x).$$An operator with the above form will be called a \emph{pseudo-differential operator} with symbol $\sigma_T (x,\xi)$. When the associated symbol does not depends on the $\Z_p$- variable, that is $\sigma_T (x , \xi) = \sigma_T (\xi)$, we will call $T_\sigma$ a \emph{Fourier multiplier} or an \emph{invariant operator}.

Our goal now is classify densely defined linear operators, which from now on will be thought as pseudo-differential operators, in terms of its associated symbol. With that purpose in mind in what follows, instead of consider abstract linear operators and find its associated symbol, we will only consider pseudo-differential operators defined in terms of a previously given symbol. Specially, we will work with pseudo-differential operators whose symbols belong to a certain H{\"o}rmander class, which we will define in the present work. These classes are thought in such a way that they include the Vladimirov operator in the ball $D^s$, $s>0$, here defined as the following Fourier multiplier $$D^s f (x):= \sum_{\xi \in \widehat{\Z}_p} \Big(|\xi |_p^s +  \frac{1- p^{-1}}{1 - p^{-(s + 1)}} (1 - \delta_{\xi ,o}) \Big)\widehat{f}(\xi) \chi_p (\xi x).$$See \cite{zuniga2, p-adicball1, p-adicball2, pseudosvinlekin, zuniga1} for more information about the Vladimirov-Taibleson operator and \cite{Onneweer1977, Onneweer1978, Dyadic, harmonicfractalanalysis} for its relation with the Gibbs derivative.
\begin{defi}\normalfont
Throughout this paper we will use sometimes the symbol ``$a \lesssim b$'' to indicate that the quantity ``$a$'' is less or equal than a certain constant times the quantity ``$b$''. Also, if $E,F$ are Banach spaces, we denote by $\mathcal{L}(E,F)$ the collection bounded linear operators $T:E \to F$. $\mathfrak{K}(E,F)$ will denote the collection of compact linear operators $T:E \to F$.
\end{defi}
\section{\textbf{Vladimirov operator and harmonic analysis on $\Z_p$}}
The study of the Vladimirov operator on a $p$-adic ball was initiated on \cite{fisicapadica}, and is also considered by Kochubei in \cite{p-adicball1, p-adicball2} restricting the action of the Vladimirov operator on functions supported in a ball $B_N:=\{x \in \Q_p : |x|_{p} \leq p^{N} \}$. In  \cite{kochubeibook} a probabilistic interpretation of this operator was given. Also, in \cite{2002J} ultrametric diffusion models constrained by hierarchical energy landscapes were analysed in terms of a master equation on a $p$-adic ball, using a different approach to that of Vladimirov on \cite{fisicapadica}, studying the modified singular integral operator $$D^s_N f (x) := \int_{B_N} \frac{f(y) - f(x)}{|y-x|^{s+1}_p} dy. $$This operator is distinct to the restricted Vladimirov operator on $B_N$, but it is more natural in the application of $p$-adic analysis to physics and other areas \cite{2002J, Avetisov2014, 2009JPhA...42h5003A} and possess an important feature: its eigenfunctions coincide with the $p$-adic characters related to the Fourier analysis and the Pontryagin duality of the compact group $B_N$. In this case one can use the unitary irreducible representations of $B_N$, which are the characters of the group, to perform a Fourier analysis as in \cite{nonharm, Ruzhansky2018} with the model operator $D^s$ and moreover, this analysis allows one to define suitable pseudo-differential operators similar to the toroidal case, besides the lack of derivatives (in a classical sense) on the $p$-adic variable. For these reasons we will work with the Vladimirov operator in the sense of \cite{p-adicbasis} acting on functions defined, by a matter of simplicity, on the $p$-adic unit ball $\Z_p$, also known as the compact group of $p$-adic integers.

\begin{defi}[Vladimirov operator]\normalfont\label{defvladimirov}
Let $s>0$. We will call \emph{Vladimirov operator} to the linear operator $D^s$ acting on functions defined on $\Z_p$ by the formula $$D^s f (x) := - \frac{1}{\Gamma_p (- s)}\int_{\Z_p} \frac{f(y) - f(x)}{|y-x|^{s+1}_p} dy,$$ where $|\cdot|_p$ is the usual $p$-adic absolute value, $\Gamma_p$ is the $p$-adic gamma function given by formula $$\Gamma_p (-s):= \frac{1-p^{-s-1}}{1-p^{s}},$$ and $dy$ is the normalised Haar measure on $\Z_p$. 
\end{defi}
As it is discussed on \cite[Lemma 1]{p-adicbasis}, the eigenfunctions of the above defined operator operator are given by the $p$-adic characteres $\chi_p (\xi x)$ with $\xi  \in \widehat{\Z}_p \cong  \Q_p / \Z_p$, and corresponding eigenvalues $$|\xi |_p^s +  \frac{1- p^{-1}}{1 - p^{-(s + 1)}} (1 - \delta_{\xi ,o}),$$where $\delta_{\xi ,o}$ is the usual Kronecker delta.  
As we know, any continuous additive character of $\Q_p$ has the form $x \to \chi_p (\xi x)$, $\xi  \in \Q_p$. By the duality theorem (see for example \cite[Theorem 27]{pontryagin}) the dual group $\widehat{\Z}_p$ is isomorphic to the discrete group $\Q_p / \Z_p $ consisting of the cosets $$ p^{-m}(r_0 + ...+ r_{m-1} p ^{m-1}) + \Z_p, \esp \esp r_j \in \{0,...,p-1 \}, \esp m \in \N_0.$$ Analytically, this isomorphism means that any nontrivial continuous character of $\Z_p$ has the form $\chi_p (\xi x), x \in \Z_p$, where $|\xi |_p >1$ and $\xi  \in \Q_p$ is considered as a representative of the class $\xi  + \Z_p$. Note that $|\xi |_p$ does not depend on the choice of a representative of the class. By the Peter-Weyl theorem \cite[Theorem 7.5.14]{ruzhansky1} the characters form an orthonormal basis of $L^2 (\Z_p)$, and then provide us the decomposition of $L^2 (\Z_p)$ as a direct sum of a sequence of one-dimensional sub-spaces, needed to have a notion of Fourier analysis.

\section{\textbf{Symbol classes and amplitudes}} By the Peter-Weyl theorem we have a decomposition of $L^2 (\Z_p)$ as the direct sum of the eigenspaces of the Vladimirov operator associated to a given eigenvalue. However, for simplicity, we will think on $L^2 (\Z_p)$ as the direct sum of the one dimensional spaces spanned by each $p$-adic character. The Fourier transform and its inverse in this context are given by $$ \widehat{f} (\xi) = \mathcal{F}_{\Z_p} f (\xi ) := \int_{\Z_p} f(x) \overline{\chi_p ( \xi  x)} dx, \esp \esp \text{and} \esp \esp \mathcal{F}_{\Z_p}^{-1} \varphi (x) := \sum_{\xi  \in\widehat{\Z_p}} \varphi (\xi ) \chi_p (\xi x),$$ where $dx$ is the  normalised Haar measure on $\Z_p$. We will measure the growth of the Fourier coefficients of a function by comparison with the weight $\langle \xi \rangle := \max\{1,|\xi |_p \}$, and then the corresponding definition of Sobolev spaces and smooth function, from the perspective of \cite{Ruzhansky2018}, is:
 
 \begin{defi}\label{defp-adicsobolev}\normalfont
 For our purposes the \emph{Sobolev spaces} $H^s (\Z_p)$ are the metric completion of $$\hb^\infty :=Span\{\chi_p (\xi x)\}_{\xi  \in \widehat{\Z}_p},$$ with the norm $$||f||_{H^s (\Z_p)} := ||\langle \xi  \rangle^{s} \widehat{f} (\xi )||_{\ell^2 (\widehat{\Z}_p)}.$$ We will call \emph{smooth functions over} $\Z_p$ to the elements of the class$$C^\infty (\Z_p) := \bigcap_{s \in \R} H^s (\Z_p).$$The class of \emph{distributions on} $\Z_p$ is defined as $$\mathfrak{D}(\Z_p) := \bigcup_{s \in \R} H^s (\Z_p).$$ The class $C^{\infty} (\Z_p \times \Z_p)$ is defined as $$C^{\infty} (\Z_p \times \Z_p):=\Big\{g \in L^2 (\Z_p \times \Z_p) \esp : \Big|\esp \int_{\Z_p \times \Z_p} g(x,y) \overline{\chi_p(\xi x)} \overline{\chi_p(\eta y)} dy dx\Big| \leq C_{g, s_1,s_2 } \langle \xi \rangle^{s_1} \langle \eta \rangle^{s_2} \Big\},$$for all $s_1 , s_2 \in \R$. We will also use the notation $$J_s f (x):= \sum_{\xi \in \widehat{\Z}_p} \langle \xi \rangle^s \widehat{f}(\xi) \chi_p (\xi x).$$
 \end{defi}
 
For the above defined Sobolev spaces we can prove the following embedding lemma:

\begin{lema}\label{p-adicl1Hsnorms}
Let $s > 1/2$ be a given real number. Then $H^s (\Z_p)$ is contained in $L^\infty (\Z_p)$ with continuous inclusion.
\end{lema}

\begin{proof}
Just notice that by Plancherel formula we have  $$||\widehat{f}||_{\ell^{{2}} (\widehat{\Z}_p)} = || f||_{L^2 (\Z_p)}.$$ So we get  \begin{align*}
    ||f||_{L^\infty (\Z_p)} \leq ||\widehat{f}||_{\ell^1 (\widehat{\Z}_p)} &\leq \big( \sum_{\xi \in \widehat{\Z}_p} \langle \xi \rangle^{-2s} \big)^{1/2}||  \langle \xi \rangle^s \widehat{f}||_{\ell^{{2}} (\widehat{\Z}_p)}\\ & = \big( \sum_{\xi \in \widehat{\Z}_p} \langle \xi \rangle^{-2s} \big)^{1/2} || f||_{H^s (\Z_p)},
\end{align*}which conclude the proof.
\end{proof}
We state here the discrete version of Taylor expansion because of its relevance for the rest of the work.

\begin{teo}[Discrete Taylor expansion on $\Z$]
Let $\varphi: \Z \to \C$ be any given function. Then for $u,v \in \Z$ we can write $$\varphi(u+v) = \sum_{n < M} \frac{1}{n!} v^n \Delta^n \varphi (u) + r_M (u,v),$$with the remainder satisfying $$|\Delta^s r_M (u,v)| \leq C_M \max_{\nu \in Q(v)} |v^M \Delta^{M+s} \varphi (u + \nu)| , $$where $Q(v):=\{\nu \in \Z \esp : \esp |\nu| \leq |v|\}$, and $\Delta \varphi (u) := \varphi (u+1) - \varphi (u)$ is the usual difference operator. See \cite{ruzhansky1} for more details.
\end{teo}Also, we state an inequality that will be used along this work.

\begin{pro}[Peetre inequality]
For every $\xi , \eta \in \widehat{\Z}_p$ and every $s \in \R$ we have $$\langle \xi +\eta \rangle^s \lesssim \langle \xi \rangle^{|s|} \langle \eta \rangle^s.$$
\end{pro}

As we already mentioned in the $\Z_p$ setting we can put linear operators in terms of a symbol as in the following definition. 

\begin{defi}\normalfont
Given a densely defined linear operator $$T: \hb^\infty  \subset Dom(T) \subseteq L^2 (\Z_p) \to L^2 (\Z_p),$$ we define its associated symbol as $$\sigma_T (x,\xi ) := \overline{ \chi_p (\xi x)} T (\chi_p (\xi x)).$$
\end{defi} 

Now we want to show that it is possible to define a pseudo-differential calculus on $\Z_p$, analogous to the case of $\Td$, for a certain class of linear operators. For this we need suitable symbols and difference operators so, we proceed to introduce some definitions.

\begin{defi}\label{p-adicdifference}\normalfont
Given a function $\varphi : \Q_p \to \C$ in the form $\varphi ( x) = f (|x|_p)$, where $f$ is a function defined on all $\Z$, let us define the difference operator $\Delta$ acting on functions $\varphi$ as the usual difference operator$$\Delta \varphi (x) := f (|x|_p  + 1) - f (|x|_p)= \Delta f \circ | \cdot |_p (x).$$ Also, for functions $\sigma : \Z_p \times \widehat{\Z}_p \to \C$, we define the ``derivative operator" $\partial_x^h$ by the formula $$\partial_x^h \sigma (x,\xi ) := \sum_{\eta \in \widehat{\Z}_p} (|\eta-\xi |_p - |\xi |_p)^h \widehat{\sigma} (\eta,\xi ) \chi(\eta x).$$ 
\end{defi}
\begin{rem}
For every $\xi \in \widehat{\Z}_p$ the following inequality holds: $$||\partial^h_x \sigma(\cdot , \xi)||_{L^2 (\Z_p)} \lesssim||D^h \sigma (\cdot , \xi)||_{L^2 (\Z_p)}.$$
\end{rem}
We proceed then to introduce now two new versions of the Hörmander classes corresponding to $\Z_p$.

\begin{defi}\label{p-adicHOrmanderclasses} \normalfont
Let $0 \leq \delta < \rho \leq 1$ be given real numbers. We define the Hörmander classes $S^m_{\rho, \delta} (\Z_p \times \widehat{\Z}_p)$ as the collection of complex valued functions $\sigma (x , |\xi |_p)$ defined on $\Z_p \times \widehat{\Z}_p$ that satisfy the following estimate $$| \Delta_\xi^\alpha D_x^\beta \sigma (x,|\xi |_p) | \leq C_{\sigma,m, \alpha, \beta}\langle \xi  \rangle^{m -\rho \alpha +  \delta \beta},$$ for every $\alpha, \beta \in \N_0.$ We will denote by $Op(S^m_{\rho, \delta} (\Z_p \times \widehat{\Z}_p))$ the class of pseudo-differential operators with symbol in the H{\"o}rmander class $S^m_{\rho, \delta} (\Z_p \times \widehat{\Z}_p)$. Furthermore, we define $$S^{-\infty} (\Z_p \times \widehat{\Z}_p) := \bigcap_{m \in \R} S^m_{0,0} (\Z_p \times \widehat{\Z}_p),$$ $$S^\infty_{\rho, \delta} (\Z_p \times \widehat{\Z}_p) := \bigcup_{m \in \R} S^m_{\rho, \delta} (\Z_p \times \widehat{\Z}_p).$$Sometimes we will denote by $Op(S^\infty_{\rho, \delta} (\Z_p \times \widehat{\Z}_p))$ the class of pseudo-differential operators with symbol in the class $S^\infty_{\rho, \delta} (\Z_p \times \widehat{\Z}_p)$
\end{defi}
\begin{rem}
There already exists a notion of symbol classes on $\Z_p$ which is a particular case of the definition given by L. Saloff-Coste in \cite{pseudosvinlekin} and Weiyi Su in \cite{harmonicfractalanalysis, padic2} for Vilenkin groups. The above definition of H{\"o}rmander classes intend to define a symbolic calculus in the same way as \cite{ruzhansky1} and, as we will see in Section 5, allow us to define asymptotic expansions for symbols, among other features. However, we have to point to the fact that this definition does not consider the special properties given by the ultrametric induced by the $p$-adic absolute value $|\cdot|_p$, in contrast with the definitions of L. Saloff-Coste and Weiyi Su. So, even when is reasonable to expect some similarity between the theory on $\Z_p$ and the theory on $\T^d$, both compact abelian groups, it is also reasonable to find essential differences because of the dissimilarity on its topologies. This important observation motivates the following second definition of H{\"o}rmander classes that, unlike the first, considers the special properties of $\Z_p$ as the definition of Saloff-Coste and Weiyi Su. We affirm that this definition is appropriate due to the relationship that exists between the classes defined below and some infinite matrix algebras studied by Gr{\"o}chenig \cite{Grochenig2010, Grochenig2010, Grochenig2010wiener} and S. Jaffard \cite{Jaffard}. We will discuss it in detail in Section 6.      
\end{rem}
\begin{defi}\label{hormanderclasses2}\normalfont
    \esp
  \begin{enumerate}
    \item[(i)] For functions $\varphi: \widehat{\Z}_p \to \C$ let us define the difference operator $\triangleplus$ as $$\triangleplus_\eta^\xi \varphi (\xi) := \varphi(\xi + \eta) - \varphi (\xi).$$
    \item[(ii)] Let $m \in \R$ and $0 \leq \delta \leq \rho \leq 1$ be given real numbers. We define the symbol classes $\Tilde{S}^m_{\rho, \delta} (\Z_p \times \widehat{\Z}_p)$ as the collection of measurable functions $\sigma: \Z_p \times \widehat{\Z}_p \to \C$ such that, for all $\alpha , \beta \in \N_0$, the estimate $$|D^\beta_x \triangleplus_\eta^\xi \sigma (x, \xi )| \leq C_{\sigma, m, \alpha, \beta} |\eta|^\alpha_p \langle \xi \rangle^{m - \rho \alpha + \delta \beta } , $$ holds for some $C_{\sigma, m, \alpha, \beta,}>0$, and any $(x, \xi,\eta) \in \Z_p \times \widehat{\Z_p} \times \widehat{\Z_p},$ $|\eta|_p \leq \langle \xi \rangle.$ As before, we will denote by $Op(\Tilde{S}^m_{\rho, \delta} (\Z_p \times \widehat{\Z}_p))$ the class of pseudo-differential operators with symbol in the H{\"o}rmander class $\tilde{S}^m_{\rho, \delta} (\Z_p \times \widehat{\Z}_p)$. Furthermore, we define$$\Tilde{S}^{-\infty} (\Z_p \times \widehat{\Z}_p) := \bigcap_{m \in \R} \Tilde{S}^m_{0,0} (\Z_p \times \widehat{\Z}_p),$$ $$\Tilde{S}^\infty_{\rho,\delta} (\Z_p \times \widehat{\Z}_p) := \bigcup_{m \in \R} \Tilde{S}^m_{\rho, \delta} (\Z_p \times \widehat{\Z}_p).$$Sometimes we will denote by $Op(\Tilde{S}^\infty_{\rho, \delta} (\Z_p \times \widehat{\Z}_p))$ the class of pseudo-differential operators with symbol in the class $\Tilde{S}^\infty_{\rho, \delta} (\Z_p \times \widehat{\Z}_p)$ 
  \end{enumerate}
 \end{defi}
\begin{rem}
We can easily se that $ S^m_{0 ,0} (\Z_p \times \widehat{\Z}_p) \subset \Tilde{S}^m_{0, 0} (\Z_p \times \widehat{\Z}_p) $.
\end{rem}
Observe that the above defined symbol classes include operators of the form $$P(x , D) := \sum_{i=1}^n a_i (x) D^{s_i},$$ where $D^{s_i}$ is the Vladimirov operator of order $s_i > 0$, and the functions $a_i$ belong to $C^{\infty} (\Z_p)$. These operators may be thought of as analogues of partial differential operators. Clearly, their associated symbol is given by  $$P(x , |\xi |_p) := \sum_{i=1}^n a_i (x) ( |\xi |_p^{s_i} + \frac{1- p^{-1}}{1 - p^{-(s + 1)}} (1 - \delta_{\xi ,o})),$$ and they define continuous operators on $C^\infty (\Z_p)$ and between Sobolev spaces. The later statement is indeed a consequence of a more general fact:

\begin{pro}\label{L2boundedness}
Let $T_\sigma$ be a pseudo-differential operator with symbol $\sigma (x , \xi)$, and let $s \in \R$ a given real number. Assume that $$|| \sigma (\cdot , \xi )||_{H^{\beta + |s|} (\Z_p)} \leq C \langle \xi \rangle^{m},$$ for $ \beta >1/2$. Then $T_\sigma$ extends to a bounded operator from $H^{s+m} (\Z_p)$ to $H^{s} (\Z_p)$.
\end{pro}

\begin{proof}
To begin with, $T_\sigma$ extends to a linear operator in $\mathcal{L}(H^{s+m} (\Z_p) , H^{s} (\Z_p^d))$ if and only if $A:=J_s T_\sigma J_{- (s+m)}$ extends to a bounded operator on $L^2 (\Z_p)$. The Symbol of $A$ is given by $$\sigma_A (x , \xi) = \frac{1}{\langle \xi \rangle^{s+m}}\sum_{\eta \in \widehat{\Z}_p} \langle \eta + \xi \rangle^{s} \widehat{\sigma} (\eta , \xi) \chi_p (\eta x). $$ Now let us define $$A_y f(x) := \sum_{\xi \in \widehat{\Z}_p} \sigma_A (y , \xi) \widehat{f}(\xi) \chi_p (\xi x ),$$ so $A_x f (x) = Af (x)$. Thus this give us $$||Af||_{L^2 (\Z_p)}^2 = \int_{\Z_p} |A_x f(x)|^2 dx \leq \int_{\Z_p} ||A_y f(x)||^2_{L^\infty_y (\Z_p)} dx,$$ which combined with the Sobolev embedding theorem leads to $$ \int_{\Z_p} ||A_y f(x)||^2_{L^\infty_y (\Z_p)} dx 
\leq \int_{\Z_p} \int_{\Z_p} |D_y^\beta A_y f(x)|^2 dy dx.$$Changing the order of integration we obtain for $f \in \hb^\infty$ \begin{align*}
    ||T_\sigma f||^2_{L^2 (\Z_p)} &\lesssim \int_{\Z_p} \int_{\Z_p} |D_y^\beta A_y f(x)|^2 dy dx \\ &=  \int_{\Z_p} \int_{\Z_p} |D_y^\beta A_y f(x)|^2 dx dy \\ &=  \int_{\Z_p}\Big( \sum_{\xi \in \widehat{\Z}_p } |D^\beta_y \sigma_A (y , \xi)|^2 |\widehat{f}(\xi)|^2\Big) dy \\ &= \sum_{\xi \in \widehat{\Z}_p } \int_{\Z_p} |D^\beta_y \sigma_A (y , \xi)|^2 dy |\widehat{f}(\xi)|^2.
\end{align*}
Observe that $$D^\beta_y \sigma_A (y , \xi) = \frac{1}{\langle \xi \rangle^{s+m}}\sum_{\eta \in \widehat{\Z}_p} | \eta |^\beta_p \langle \eta + \xi \rangle^{s} \widehat{\sigma} (\eta , \xi) \chi_p (\eta y).$$ Thus using Peetre inequality \begin{align*}
    ||D^\beta_y \sigma_A (y , \xi)||_{L^2_y (\Z_p)}^2 &= \frac{1}{\langle \xi \rangle^{2(s+m)}}\sum_{\eta \in \widehat{\Z}_p} | \eta |^{2\beta}_p \langle \eta + \xi \rangle^{2s} |\widehat{\sigma}(\eta , \xi)|^2 \\ & \lesssim   \frac{1}{\langle \xi \rangle^{2m}} \sum_{\eta \in \widehat{\Z}_p} \langle \eta \rangle^{2(\beta+|s|)}  |\widehat{\sigma}(\eta , \xi)|^2 \\ &\lesssim  \frac{1}{\langle \xi \rangle^{2m}}||D_y^{\beta + |s|} \sigma (y, \xi)||_{L^2_y (\Z_p)}^2 \leq C.
\end{align*}Then finally $$||T_\sigma f||_{L^2 (\Z_p)} \leq C ||f||_{L^2 (\Z_p)},$$concluding the proof.
\end{proof}

\begin{coro}\label{coroboundedhormanderclasses}
Let $T_\sigma$ be a pseudo-differential operator with symbol $\sigma \in \tilde{S}^m_{0, 0} (\Z_p \times \widehat{\Z}_p)$. Then $T_\sigma \in \mathcal{L}(H^{s+m} (\Z_p) , H^{s} (\Z_p))$ for every $s \in \R$.
\end{coro}

We also have a version of \cite[Theorem 4.3.1]{ruzhansky1}.
\begin{pro}[Smoothing]
The following conditions are equivalent: 
\begin{enumerate}
    \item[(i)] $T \in \mathcal{L}(H^s (\Z_p), H^t(\Z_p))$ for every $s , t \in \R$;
    \item[(ii)] $\sigma_T \in \Tilde{S}^{- \infty} (\Z_p \times \widehat{\Z}_p)$;
    \item[(iii)] There exists $K_T \in C^\infty (\Z_p \times \Z_p)$ such that for all $f \in C^\infty (\Z_p)$ we have $$T f (x) = \int_{\Z_p} K_T (x,y) f (y) dy .$$ 
\end{enumerate}
\end{pro}
\begin{proof}
\esp 
\begin{itemize}
    \item[](i) $\implies$ (ii) : If $T \in \mathcal{L}(H^s (\Z_p), H^t(\Z_p))$ for every $s , t \in \R$ then for every $m \in \R$ we have $T J_{-(m-|s|)} \in \mathcal{L}(L^2 (\Z_p), H^s (\Z_p)).$ It follows $$||TJ_{-(m-|s|)}\chi_p(\xi \cdot)||_{H^s(\Z_p)} = ||\sigma_T (\cdot , \xi) \langle \xi \rangle^{-(m-|s|)}\chi_p(\xi \cdot)]||_{H^s (\Z_p)} \lesssim ||\chi_p(\xi \cdot)||_{L^2 (\Z_p)} = 1,$$for every $\xi \in \widehat{\Z}_p$. Also, we can deduce the following inequalities: \begin{align*}
    ||\sigma_T (\cdot , \xi) \langle \xi \rangle^{-(m-|s|)} \chi_p(\xi \cdot) ]||_{H^s (\Z_p)}^2 &= ||J_s \sigma_T (\cdot , \xi) \langle \xi \rangle^{-(m-|s|)} \chi_p(\xi \cdot)]||_{L^2 (\Z_p)}^2 \\ &= \frac{1}{\langle \xi \rangle^{2(m-|s|)}} \sum_{\eta \in \widehat{\Z}_p} \langle \eta \rangle^{2s} |\widehat{\sigma}_T(\eta - \xi , \xi)|^2\\ &\gtrsim \frac{1}{ \langle \xi \rangle^{2m}} \sum_{\eta \in \widehat{\Z}_p} \langle \eta - \xi \rangle^{2s} |\widehat{\sigma}_T(\eta - \xi , \xi)|^2\\&=\frac{1}{\langle \xi \rangle^{2m}} ||\sigma_T (\cdot , \xi)||_{H^s(\Z_p)}^2.
\end{align*}Thus we conclude $$||\sigma_T(\cdot , \xi)||_{H^s(\Z_p)} \leq C \langle \xi \rangle^m, $$ for every $s,m \in \R$.
\item[](ii)$\implies$(i): Follows from Corollary \ref{coroboundedhormanderclasses}. 
\item[](ii)$\implies$(iii): Just observe that \begin{align*}
    Tf(x)&= \sum_{\xi \in \widehat{\Z}_p} \sigma_T (x,\xi) \widehat{f}(\xi) \chi_p(\xi x)\\ &\sum_{\xi \in \widehat{\Z}_p} \sigma_T (x,\xi) \Big( \int_{\Z_p} f(y) \overline{\chi_p(\xi y)} dy \Big) \chi_p(\xi x)\\& = \int_{\Z_p} \Big( \sum_{\xi \in \widehat{\Z}_p} \sigma_T (x,\xi) \chi_p(\xi (x -y))\Big) f(y) dy,
\end{align*}and one can easily see that $$K_T (x,y):=\sum_{\xi \in \widehat{\Z}_p} \sigma_T (x,\xi) \chi_p(\xi (x -y)) \in C^{\infty}(\Z_p \times \Z_p).$$
\item[] (iii)$\implies$(ii): Suppose that for some $K_T \in C^{\infty}(\Z_p \times \Z_p)$ we have $$Tf(x)=\int_{\Z_p} K_T (x,y)f(y)dy.$$Then the associated symbol $\sigma_T (x,\xi)$ of $T$ is $$\sigma_T (x, \xi) = \overline{\chi_p(\xi x)} \int_{\Z_p} K_T (x,y) \overline{\chi_p(-\xi y)} dy=\overline{\chi_p(\xi x)} \mathcal{F}_{\Z_p}^y K_T (x , - \xi).$$Next $$\widehat{\sigma}_T (\eta , \xi) = \mathcal{F}_{\Z_p}^x\mathcal{F}_{\Z_p}^y K_T (\eta + \xi , - \xi),$$which implies \begin{align*}
    \sum_{\eta \in \widehat{\Z}_p} \langle \eta \rangle^{2s} |\widehat{\sigma}_T (\eta , \xi)|^2 &= \sum_{\eta \in \widehat{\Z}_p} \langle \eta \rangle^{2s} |\mathcal{F}_{\Z_p}^x\mathcal{F}_{\Z_p}^y K_T (\eta + \xi , - \xi)|^2\\&=\sum_{\eta \in \widehat{\Z}_p} \langle \eta + \xi - \xi \rangle^{2s} |\mathcal{F}_{\Z_p}^x\mathcal{F}_{\Z_p}^y K_T (\eta + \xi , - \xi)|^2 \\&\lesssim \langle \xi \rangle^{2|s|}\sum_{\eta \in \widehat{\Z}_p} \langle \eta + \xi \rangle^{2s} |\mathcal{F}_{\Z_p}^x\mathcal{F}_{\Z_p}^y K_T (\eta + \xi , - \xi)|^2 \\ &= \langle \xi \rangle^{2|s|} \sum_{\eta \in \widehat{\Z}_p} \langle \eta \rangle^{2s} |\mathcal{F}_{\Z_p}^x\mathcal{F}_{\Z_p}^y K_T (\eta, - \xi)|^2 \\ & \leq C \langle \xi \rangle^{m},
\end{align*}for every $s,m \in \R$, and thus $\sigma_T \in \Tilde{S}^{-\infty} (\Z_p \times \widehat{\Z}_p).$
\end{itemize}
\end{proof}

Our next step is the development of a pseudo-differential calculus over $\Z_p$. For doing this in the next section, we introduce an important object that will allow us to prove an adjoint fomula similar to \cite[Theorem 4.7.7]{ruzhansky1}. 
\begin{defi}\normalfont
An \emph{amplitude} is a measurable function $a: \Z_p \times \Z_p \times \widehat{\Z}_p \to \C$ such that $a(\cdot , \cdot , \xi ) \in C^\infty (\Z_p \times \widehat{\Z}_p)$ for each $\xi  \in \widehat{\Z}_p$. For an amplitude $a$ its associated operator is defined as $$T_a f (x) := \sum_{\xi  \in \widehat{\Z}_p } \int_{\Z_p} \chi ((x-y)\xi ) a (x ,y, \xi ) f (y) dy.$$It will be convenient to work with amplitudes in special classes, similar to the Hörmander classes, so we define the amplitude classes $S^m_{\rho ,  \delta} (\Z_p^2 \times \widehat{\Z}_p)$ as the collection of amplitudes $a (x, y , |\xi|_p)$ satisfying $$|\Delta_\xi^\alpha D_x^\beta D_y^\gamma a (x,y,|\xi|_p)| \leq C_{a,m, \alpha, \beta, \gamma} \langle \xi \rangle^{m -\rho \alpha + \delta (\beta + \gamma )},$$for every $\alpha, \beta, \gamma \in \N_0$.
\end{defi}
\section{\textbf{Symbolic calculus} I}

To develop our first version of symbolic calculus we begin with asymptotic sums.

\begin{teo}[Asymptotic sum of symbols]\label{asymptotic sum}
Let $(m_j)_{j \in \N_0}$ be a strictly decreasing sequence of real numbers such that $m_j \to - \infty $ as $|\xi|_p \to \infty$. Let $(\sigma_j)_{j \in \N_0}$, $\sigma_j \in S^{m_j}_{\rho , \delta} (\Z_p \times \widehat{\Z}_p)$, be a sequence of $p$-adic symbols. Then there exists a $p$-adic symbol $\sigma \in S^{m_0}_{\rho ,  \delta} (\Z_p \times \widehat{\Z}_p)$ such that for all $N \in \N_0$ $$\sigma \esp \esp \mysim \esp \esp \sum_{j=0}^{N-1}  \sigma_j,$$ where the above means $$\sigma - \sum_{j=0}^{N-1}  \sigma_j \in S^{m_N}_{\rho , \delta} (\Z_p \times \widehat{\Z}_p).$$ 
\end{teo}
\begin{proof}
Let us define the function $\varphi : \Z \to \C$  by \[\varphi (n) := \begin{cases}
0 & \esp \text{if} \esp |n|\leq 1, \\ 1 & \esp \text{if} \esp |n|>1,
\end{cases}\] and the sequence of functions $(\varphi_j)_{j \in \N_0}$ defined on $\Q_p$ by $\varphi_j( \xi) = \varphi(p^{-j} |\xi|_p)$. It is easy to see that each $\Delta^n_\xi \varphi_j$ has bounded support, so that by the discrete Leibniz formula \cite[Lemma 3.3.6]{ruzhansky1} we get $$|\Delta^\alpha_\xi \partial^\beta_x [\varphi_j (\xi) \sigma_j (x,\xi)]|\leq C_{j  \alpha \beta} \langle \xi \rangle^{m_j - \rho\alpha - \delta \beta},$$ for some positive constants $C_{j \alpha \beta}$, since $\sigma_j \in S^{m_j}_{\rho , \delta} (\Z_p \times \widehat{\Z}_p)$. We also note that for any $\xi \in \widehat{\Z}_p$ fixed $\Delta_\xi^\alpha[\varphi_j (\xi) \sigma_j (x,\xi)]$ vanishes when $j$ is large enough. This justifies the definition $$\sigma (x,\xi) := \sum_{j \in \N_0} \varphi_j (\xi) \sigma_j (x,\xi),$$ where clearly $\sigma \in S^{m_0}_{\rho , \delta} (\Z_p \times \widehat{\Z}_p )$. Furthermore \begin{align*}
    |\Delta^\alpha_x \partial^{\beta}_\xi [\sigma (x , |\xi|_p) - &\sum_{j=0}^{N-1}\sigma_j (x , |\xi|_p)]| \\&\leq  \sum_{j=0}^{N-1} |\Delta^\alpha \partial^\beta [(\varphi_j (\xi) - 1) \sigma_j (x , |\xi|_p)]| + \sum_{j=0 }^{\infty} |\Delta^\alpha \partial^\beta [\varphi_j (\xi)  \sigma_j (x , |\xi|_p)]|.
\end{align*}The first part of the above sum vanishes for large $|\xi|_p$, so it can be bounded by $C_{r N \alpha \beta} \langle \xi \rangle^{-r}$ for any $r \in \R$. The second part is majorised by $C_{N\alpha \beta} \langle \xi \rangle^{m_N - \rho \alpha + \delta \beta}$. This concludes the proof.
\end{proof}

\begin{defi}\normalfont
The formal series $\sum_{j=0}^{\infty} \sigma_j$ in Theorem \ref{asymptotic sum} is called an asymptotic expansion of the symbol $\sigma \in S^{m_0}_{\rho , \delta} (\Z_p \times \widehat{\Z}_p)$. In this case we write $$\sigma \sim \sum_{j=0}^{\infty} \sigma_j.$$
\end{defi}

Combining asymptotic expansions and the definition of amplitude operators we can give explicit formula for the symbol of the transpose and the adjoint of pseudo-differential operators.

\begin{pro}[Symbols of amplitude operators]\label{symbols p-adic amplitudes}
Let $0 \leq \rho < \delta \leq 1$. For every $p$-adic amplitude $a \in S^m_{\rho , \delta} (\Z_p^2 \times \widehat{\Z}_p)$ there exists a unique $p$-adic symbol $\sigma \in S^m_{\rho , \delta} (\Z_p \times \widehat{\Z}_p)$ satisfying $T_a = T_\sigma$, and $\sigma$ has the following asymptotic expansion:$$\sigma(x , |\xi|_p) \sim \sum_{\gamma \geq 0} \frac{1}{\gamma!} \Delta^\gamma_\xi \partial^\gamma_y a (x , y, |\xi|_p)|_{y=x}$$
\end{pro}
\begin{proof}
Clearly, amplitude operators are densely defined linear operators whose symbol can be calculated as \begin{align*}
    \sigma (x , |\xi|_p)  = \overline{\chi_p(\xi x)} T_a (\chi_p(\xi x)) &= \sum_{\eta \in \widehat{\Z}_p} \int_{\Z_p} \chi_p((\eta-\xi)(x-y)) a(x,y,|\eta|_p) dy\\&= \sum_{\eta \in \widehat{\Z}_p} \chi_p((\eta-\xi)x) \mathcal{F}_y a (x , \eta - \xi , |\eta|_p )\\&= \sum_{\eta \in \widehat{\Z}_p} \chi_p(\eta x) \mathcal{F}_y a (x , \eta  , |\xi + \eta|_p ).
\end{align*}Now we apply the discrete Taylor formula to obtain \begin{align*}
    \sigma (x,|\xi|_p) &= \sum_{\eta \in \widehat{\Z}_p} \chi_p (\eta x) \sum_{\gamma < N} \frac{1}{\gamma!} \Delta^\gamma_\xi (|\xi+\eta|_p -|\xi|_p)^\gamma \mathcal{F}_y a (x , \eta, |\xi|_p) + \sum_{\eta \in \widehat{\Z}_p} \chi_p (\eta x) R_N (x,\eta,\xi)\\&= \sum_{\gamma < N} \frac{1}{\gamma!} \Delta^\alpha_\xi \partial^\beta_y a(x,y,|\xi|_p) + \sum_{\eta \in \widehat{\Z}_p} \chi_p(\eta x) R_N (x,\eta ,\xi).
\end{align*}Finally we can estimate \begin{align*}
    |\Delta^{\alpha'}_\xi D^{\beta'}_x R_N (x,\eta ,\xi)| &\leq \frac{1}{N!} ||\xi-\eta|_p - |\xi|_p|^N \max_{|w|=N, \esp \nu \in Q(|\xi-\eta|_p - |\xi|_p)} |\Delta^{\alpha' + w} \partial^{\beta'} \mathcal{F}_y a (x,\eta,|\xi|_p + \nu)| \\ & \leq C_{\alpha' \beta' m N} \langle \eta \rangle^{N-r} \langle \xi \rangle^{m - N -\rho \alpha' + \delta \beta' },
\end{align*} for $r$ large enough, finishing the proof.
\end{proof}
\begin{pro}[Transpose operators.]
Let $0 \leq \delta < \rho \leq 1.$ Let $T_\sigma$ be a pseudo-differential operator with symbol $\sigma \in S^m_{\rho , \delta} (\Z_p \times \widehat{\Z}_p)$. Then the transpose operator $T^t_\sigma$ is a pseudo-differential operator with symbol $\sigma^t \in S^m_{\rho , \delta} (\Z_p \times \widehat{ \Z_p})$, and it has the following asymptotic expansion $$\sigma^t (x,|\xi|_p) \sim \sum_{\gamma \geq 0} \frac{1}{\gamma!} \Delta^\gamma_\xi \partial^\gamma_y \sigma (y,|\xi|_p)|_{y=x}.$$
\end{pro}
\begin{proof}
Consider the linear operator $T_a$ associated to the amplitude $a (x,y,|\xi|_p) := \sigma (x,|\xi|_p)$. We get \begin{align*}
    \int_{\Z_p} v(x) T_\sigma^t u (x) dx &= \int_{\Z_p} u(y) T_\sigma v(y) dy \\&= \int_{\Z_p} u(y) \Big(\sum_{\xi \in \widehat{\Z}_p} \int_{\Z_p} \chi_p(\xi(y-x)) a (y,x,|\xi|_p) v(x) dx \Big) dy \\ &=\int_{\Z_p} v(x) \Big(\sum_{\xi \in \widehat{\Z}_p} \int_{\Z_p} \chi_p(\xi(y-x)) a (y,x,|\xi|_p) u(y) dy \Big) dx.
\end{align*}Thus $T_\sigma^t$ is a linear operator $T_{a^t}$ associated to the amplitude $$a^t (x,y,|\xi|_p) = a (y,x,|\xi|_p) = \sigma (y,|-\xi|_p)= \sigma (y,|\xi|_p).$$Finally, by Proposition \ref{symbols p-adic amplitudes} $T_{a^t}$ is a pseudo-differential operator with symbol $\sigma^t \in S^{m}_{\rho, \delta} (\Z_p \times \widehat{\Z}_p)$ and it has the following asymptotic expansion $$\sigma^t (x,|\xi|_p) \sim \sum_{\gamma \geq 0} \frac{1}{\gamma!} \Delta^\gamma_\xi \partial^\gamma_y \sigma (y,|\xi|_p)|_{y=x}.$$
\end{proof}

\begin{rem}
Notice the peculiarity in this case: the symbol $\sigma (x, |\xi|_p)$ and the symbol of its transpose operator $\sigma^t (x , |\xi|_p)$ have the same asymptotic expansion.
\end{rem}

\begin{coro}
Let $T_\sigma$ be a pseudo-differential operator with symbol $\sigma \in S^m_{\rho , \delta} (\Z_p \times \widehat{\Z}_p)$. Then its adjoint operator $T_\sigma^*$ is a pseudo-differential operator $T_{ \sigma^*}$ with symbol $\sigma^* \in S^m_{\rho , \delta} (\Z_p \times \widehat{\Z}_p)$, and it has the following asymptotic expansion $$\sigma^* (x,|\xi|_p) \sim \sum_{\gamma \geq 0} \frac{1}{\gamma!}\Delta^\gamma_\xi \partial^{\gamma}_y \overline{\sigma(y,|\xi|_p)}|_{y=x}.$$
\end{coro}

The next step is to give a explicit formula for composition of pseudo-differential operators. For that end we will need a couple of simple auxiliary propositions. 
\begin{pro}\label{normofproduct}
Let $s>0$ be a given real number. let $f, g \in L^2 (\Z_p)$ be functions such that $f \in H^s(\Z_p)$ and $ \widehat{J_s g} \in \ell^1( \widehat{\Z_p})$. Then the pointwise product $f g \in H^s(\Z_p)$ and $$||fg||_{H^s (\Z_p)} \leq 2^s ||f||_{H^s (\Z_p)} ||\widehat{J_s g}||_{\ell^1 (\widehat{\Z}_p)}.$$ 
\end{pro}
\begin{proof}
The Fourier series of $(fg)(x)$ is $$(fg)(x) = \sum_{\xi \in \widehat{\Z_p}} \Big( \sum_{\eta \in \widehat{\Z}_p} \widehat{f}(\xi - \eta) \widehat{g}(\eta) \Big) \chi(\xi x).$$Thus using Minkowski integral inequality \begin{align*}
    ||fg||_{H^s (\Z_p)} = ||J_s(fg)||_{L^2 (\Z_p)} &= \Big( \sum_{\xi \in \widehat{\Z_p}} \langle \xi \rangle^{2s}\Big| \sum_{\eta \in \widehat{\Z}_p} \widehat{f}(\xi - \eta) \widehat{g}(\eta) \Big|^2  \Big)^{1/2} \\ &\leq \sum_{\eta \in \widehat{\Z}_p} \Big( \sum_{\xi \in \widehat{\Z_p}} \langle \xi \rangle^{2s} |\widehat{f}(\xi - \eta) \widehat{g}(\eta)|^2 \Big)^{1/2} \\ & \leq 2^s \sum_{\eta \in \widehat{\Z}_p} \langle \eta \rangle^s \widehat{g}(\eta) \Big( \sum_{\xi \in \widehat{\Z_p}} \langle \xi - \eta \rangle^{2s} |\widehat{f}(\xi - \eta) |^2 \Big)^{1/2}\\&= 2^s ||f||_{H^s (\Z_p)} ||\widehat{J_s g}||_{\ell^1 (\widehat{\Z}_p)}.  
\end{align*}This concludes the proof.
\end{proof}

\begin{pro}
Let $\sigma_1, \sigma_2$ be symbols in the Hörmander classes $S^{m_1 }_{\rho ,0 } (\Z_p \times \widehat{\Z}_p)$ and $S^{m_2 }_{\rho ,0 } (\Z_p \times \widehat{\Z}_p)$ respectively. Then $\sigma_1 \sigma_2 \in S^{m_1 + m_2}_{\rho , 0 } (\Z_p \times \widehat{\Z}_p)$.  
\end{pro}
\begin{proof}
By discrete product rule we get $$\Delta_\xi^\alpha \sigma_ 1 (x,|\xi|_p) \sigma_2 (x ,|\xi|_p) = \sum_{l \leq \alpha } { {\alpha}\choose{l}} \Delta_\xi^l \sigma_ 1 (x,|\xi|_p) \Delta_\xi^{\alpha}\sigma_ 2 (x,|\xi|_p).$$Using Proposition \ref{normofproduct} we get \begin{align*}
    ||\Delta_\xi^\alpha \sigma_ 1 (x,|\xi|_p) \sigma_2 (x ,|\xi|_p)||_{H^s(\Z_p)} &= ||\sum_{l \leq \alpha } { {\alpha}\choose{l}} \Delta_\xi^l \sigma_ 1 (x,|\xi|_p) \Delta_\xi^{\alpha-l}\sigma_ 2 (x,|\xi|_p)||_{H^s(\Z_p)}\\ & \leq \sum_{l \leq \alpha } { {\alpha}\choose{l}} ||\Delta_\xi^l \sigma_ 1 (x,|\xi|_p) \Delta_\xi^{\alpha-l}\sigma_ 2 (x,|\xi|_p)||_{H^s(\Z_p)}\\&\leq \sum_{l \leq \alpha } { {\alpha}\choose{l}} ||\Delta_\xi^l \sigma_ 1 (x,|\xi|_p)||_{H^s (\Z_p)} ||\widehat{J_s \Delta_\xi^{\alpha-l}\sigma_ 2 (x,|\xi|_p)}||_{\ell^1 (\widehat{\Z}_p)} \\ & \leq C \sum_{l \leq \alpha } { {\alpha}\choose{l}} \langle \xi \rangle^{m_1 - \rho l} \langle \xi \rangle^{m_2 - \rho(\alpha - l)} \\ & \leq C \langle \xi \rangle^{m_1 + m_2 - \rho \alpha }.
\end{align*}This finish the proof.
\end{proof}
Now we have enough tools to prove a first version of the composition formula:
\begin{pro}[Composition formula]\label{composition formula}
Let $T_{\sigma_1} \in Op (S^{m_1}_{\rho, 0} (\Z_p \times \widehat{\Z}_p))$, $T_{\sigma_2} \in Op(S^{m_2}_{\rho, 0} (\Z_p \times \widehat{\Z}_p))$ be pseudo-differential operators. Then $T_{\sigma_1} T_{\sigma_2} = T_\sigma$, $\sigma \in Op(S^{m_1 + m_2}_{\rho, 0} (\Z_p \times \widehat{\Z}_p))$ and $\sigma$ has the following asymptotic expansion $$\sigma (x , |\xi|_p) \sim \sum_{\beta < N} \frac{1}{\beta!} \partial_x^\beta \sigma_2 (x,|\xi|_p) \Delta^\beta_\xi \sigma_2 (x, |\xi|_p). $$
\end{pro}

\begin{proof}
$$\sigma (x,|\xi|_p) = \overline{ \chi_p ( \xi x)} T_{\sigma_1} T_{\sigma_2} (\chi_p (\xi x)) = \sum_{\eta \in \widehat{\Z}_p}  \sigma_1 (x, |\eta + \xi|_p) \widehat{\sigma}_2 (\eta,|\xi|_p) \chi_p (\eta x).$$
Using the discrete Taylor formula \cite[Theorem 3.3.21]{ruzhansky1} we get \begin{align*}
    \sigma_1 (x, |\eta + \xi|_p) &= \sigma_1 (x, |\xi|_p + (|j + \xi|_p - |\xi|_p )) \\ &= \sum_{\beta < N}\frac{1}{\beta!} (|j + \xi|_p - |\xi|_p )^\beta  \Delta^\beta_\xi \sigma_1 (x, |\xi|_p) + R_N (x,|\xi|_p, |\xi-j|_p),
\end{align*} and then \begin{align*}
    \sigma (x,|\xi|)  &= \sum_{\beta < N} \sum_{\eta \in \widehat{\Z}_p}  \frac{1}{\beta!} (|\eta + \xi|_p - |\xi|_p )^\beta  \Delta^\beta_\xi \sigma_1 (x, |\xi|_p) \widehat{\sigma}_2 (\eta,|\xi|_p) + \sum_{\eta \in \widehat{\Z}_p} \chi_p(\eta x) R_N (x,|\xi|_p, |\xi-j|_p)  \\&= \sum_{\beta < N} \frac{1}{\beta!} \partial_x^\beta \sigma_2 (x,|\xi|_p) \Delta^\beta_\xi \sigma_1 (x, |\xi|_p) + E_N (x,\eta ,\xi) .
\end{align*}
Similar to \cite[Theorem 4.7.10]{ruzhansky1} we conclude that $E_N \in S^{m - N}_{\rho , 0} (\Z_p \times \widehat{\Z}_p ) $. This finish the proof.
\end{proof}
\begin{rem}
The results collected so far in this section can be summarized as follows: the class $Op(S^\infty_{\rho, \delta} (\Z_p \times \widehat{\Z}_p))$ is a filtered $*$-algebra of pseudo-differential operators. In the following section we will prove that there are other possible definitions of filtered $*$-algebras for classifying densely defined linear operators using its associated infinite matrix instead of its associated symbol.
\end{rem}
\section{\textbf{Relation with infinite matrix algebras}}
In linear algebra and functional analysis a common idea is to express linear operators by means of another mathematical object. There are usually two approaches for doing this: representing linear operators with matrices and representing them with symbols. So far we have used only the symbolic approach but in any Hilbert space $\ha$, using a Riesz basis for $\ha$, it is possible to use the matrix approach, expressing linear operators in terms of the associated matrix with respect to the given basis. In the cases where both approaches are available there should be a relation between them. In our case we are treating with linear operators acting on the Hilbert space $L^2 (\Z_p)$ where the Peter-Weyl theorem provide us an orthonormal basis so, if we look at the associated matrix of an operator instead of its symbol, we might expect to put some properties of the operator in terms of its matrix, as we have done with the symbol, and to find some relation with the symbolic approach. The purpose of this section is to make explicit that relation in the case when $\ha = L^2 (\Z_p)$. With that end let us begin by recalling the definition of the associated infinite matrix of a linear operator with respect to a given Riesz basis.
\begin{defi}[Associated infinite matrix]\normalfont\label{definfmatrix}
Given an infinite countable index set $I$, an infinite matrix indexed by $I$ is a function $M: I \times I \to \C$ with matrix entries defined by $M_{\xi \eta} := M(\xi,\eta)$, $\xi, \eta \in I$. If $M$ is an infinite matrix and $\varphi$ an infinite vector (or a function from $I$ to $\C$) then the product of the vector $\varphi$ an the matrix $M$ is defined as 
$$
    M \varphi (\xi) := \sum_{\eta \in I} M_{\xi \eta} \varphi (\eta).
$$For infinite matrices $P$ and $Q$ their product is defined as the infinite matrix with entries
\begin{align*}
    PQ_{\xi \eta} := \sum_{\gamma \in I} P_{\xi \gamma} Q_{\gamma \eta},
\end{align*}and as usual, the adjoint of the infinite matrix $M$ is the infinite matrix $M^*$ with entries

\begin{align*}
    M^*_{\xi \eta} := \overline{(M_{\eta \xi})}.
\end{align*}
Now, given a a Riesz basis $\{e_\xi\}_{\xi \in \I}$ of a Hilbert space $\ha$, with corresponding bi-orthogonal system $\{u_\xi\}_{\xi \in \I}$, and a densely defined linear operator $T: Span\{u_{\xi}\}_{\xi \in \I} \subset D(T) \subseteq \ha \to \ha$, the \emph{associated infinite matrix} $M_T$ of $T$ is the infinite matrix with entries $$(M_T)_{\xi \eta}:= (T e_\eta , u_\xi)_{\ha} , \esp \eta, \xi \in I.$$For a pseudo-differential operator $T_\sigma \in Op(\tilde{S}^m_{0, 0} (\Z_p \times \Z_p))$ we will denote its associated infinite matrix with respect to the orthonormal basis $\{\chi_p (\xi x) \}_{\xi \in \widehat{\Z}_p}$ by $M_\sigma$. 
\end{defi}
Once we have assigned an finite matrix to our linear operators, by using the boundedness of the Fourier transform from $L^2 (\Z_p)$ to $\ell^2 (\widehat{\Z}_p)$, we could think of pseudo-differential operators operators as infinite matrices acting on $\ell^2 (\widehat{\Z}_p)$. If we do that a natural question arise: \emph{is it possible to classify densely defined linear operators into a filtered $*$-algebra by means of its associated matrix?}. The answer is of course a positive answer and one method for doing that is to look at the decay properties of the matrix entries as in the work of S. Jaffard \cite{Jaffard}. This idea is well known and there are several interesting works about it. For example in \cite{Grochenig2010, Grochenig2006} K. Gr{\"o}chenig studied the following infinite matrices class:
\begin{defi}\normalfont\label{def0schurclasses}
Let $M$ be an infinite matrix indexed by $\widehat{\Z}_p$. We say that $M$ is in the \emph{Schur class} $\mathcal{S}_r (\widehat{\Z}_p)$, $r \geq 0$, if $$||M||_{\mathcal{S}_r (\widehat{\Z}_p)}:= \max \big\{ \sup_{\xi \in \widehat{\Z}_p} \sum_{\eta \in \widehat{\Z}_p} |M_{\eta \xi}| \langle \eta - \xi \rangle^r , \sup_{\eta \in \widehat{\Z}_p} \sum_{\xi \in \widehat{\Z}_p} |M_{\eta \xi}| \langle \eta - \xi \rangle^r  \big\}< \infty.$$
\end{defi}
For the Schur algebra $\mathcal{S}_r (\Z_p)$ the following properties are proven:
\begin{pro}\label{propertiesschuralgebra}
\esp
\begin{enumerate}
    \item[(i)] $\mathcal{S}_r (\widehat{\Z}_p)$ is a solid Banach $*$-algebra.
    \item[(ii)] $\mathcal{S}_r (\widehat{\Z}_p)$ is continuously embedded into $\mathcal{L}(\ell^r (\widehat{\Z}_p))$ for every $1 \leq r \leq \infty$.
    \item[(iii)] $\mathcal{S}_r (\widehat{\Z}_p)$ is inverse closed in $\mathcal{L}(\ell^2 (\widehat{\Z}_p))$.
\end{enumerate}
\end{pro}
We are now interested in the relation between the H{\"o}rmander classes that we have defined and the Schur classes. Let us begin our analysis by calculating the entries of the associated matrix of a pseudo-differential operator $T_\sigma \in Op (S^\infty_{0, 0} (\Z_p \times \widehat{\Z}_p))$. 
By Definition \ref{definfmatrix} the Associated matrix $M_\sigma$ of $T_\sigma$ has entries $$(M_\sigma)_{\eta \xi} = (T_\sigma \chi_p (\xi \cdot) , \chi_p (\eta \cdot) )_{L^2 (\Z_p)}= \int_{\Z_p} \sigma (x, \xi) \overline{\chi_p ((\eta - \xi) x)} dx = \widehat{\sigma} (\eta - \xi , \xi).$$ Thus we see that the entries of the associated matrix are given by the Fourier coefficients of the symbol. Actually we might also write the matrix entries of the associated matrix in terms of the Fourier coefficients of the symbol of the adjoint operator: \begin{align*}
    (M_\sigma)_{\eta \xi} &= (T_\sigma \chi_p (\xi \cdot) , \chi_p (\eta \cdot) )_{L^2 (\Z_p)}\\ &=( \chi_p (\xi \cdot) , T_{\sigma^*} \chi_p (\eta \cdot) )_{L^2 (\Z_p)} \\ &= \int_{\Z_p} \overline{\sigma^* (x , \eta)} \esp \overline{\chi_p ((\eta - \xi)x)} dx = \overline{\widehat{\sigma}^*(\xi - \eta , \eta)}.
\end{align*} In this way $$\sum_{\eta \in \widehat{\Z}_p} |(M_\sigma)_{\eta \xi}| \langle \eta - \xi \rangle^r = \sum_{\eta \in \widehat{\Z}_p} \langle \eta \rangle^r |\widehat{\sigma}(\eta , \xi)| = ||\widehat{J_r \sigma (\cdot , \xi)}||_{\ell^1 (\widehat{\Z}_p)},$$and$$\sum_{\xi \in \widehat{\Z}_p} |(M_\sigma)_{\eta \xi}| \langle \eta - \xi \rangle^r=\sum_{\xi \in \widehat{\Z}_p} \langle \xi \rangle^{r} |\widehat{\sigma}^*(\xi , \eta)| = ||\widehat{J_r \sigma^* (\cdot , \eta)}||_{\ell^1 (\widehat{\Z}_p)}.$$From these equalities we obtain $$||M_\sigma||_{\mathcal{S}_r (\widehat{\Z}_p)} = \max \Big\{ \sup_{\xi \in \widehat{\Z}_p} ||\widehat{J_r \sigma (\cdot , \xi)}||_{\ell^1 (\widehat{\Z}_p)}\esp , \esp \sup_{\xi \in \widehat{\Z}_p} ||\widehat{J_r \sigma^* (\cdot , \xi)}||_{\ell^1 (\widehat{\Z}_p)} \Big\}.$$
Before continuing we introduce a new definition:
\begin{defi}
We will denote by $\mathcal{S}_r(\Z_p)$ the class of densely defined linear operators $T:\hb^\infty \subset D(T) \subseteq L^2 (\Z_p) \to L^2 (\Z_p)$ such that its associated infinite matrix $M_T$ belongs to the Schur class $\mathcal{S}_r (\widehat{\Z}_p)$. That is: $$\mathcal{S}_r(\Z_p):= \{T:\hb^\infty \subset D(T) \subseteq L^2 (\Z_p) \to L^2 (\Z_p) \esp : \esp M_T \in \mathcal{S}_r (\widehat{\Z}_p) \}.$$ 
\end{defi}
With the above definition we can prove the following relation between the H{\"o}rmander class $\Tilde{S}^0_{0,0}(\Z_p \times \widehat{\Z}_p)$ and the Schur classes $\mathcal{S}_r (\widehat{\Z}_p)$:

\begin{pro}
Let $T:\hb^\infty \subset D(T) \subseteq L^2 (\Z_p) \to L^2 (\Z_p)$ be a densely defined linear operator. Then $T \in Op(\Tilde{S}^0_{0,0} (\Z_p \times \widehat{\Z_p}))$ if and only if $T \in \bigcap_{r \geq 0} \mathcal{S}_r(\Z_p)$. That is: $$Op(S^0_{0,0} (\Z_p \times \widehat{\Z}_p))= \bigcap_{r \geq 0} \mathcal{S}_r(\Z_p).$$
\end{pro}
\begin{proof}
As we will prove in Section 7 if $\sigma \in S^0_{0,0} (\Z_p \times \widehat{\Z}_p)$ then also $\sigma^* \in S^0_{0,0} (\Z_p \times \widehat{\Z_p})$ and the following estimates hold: $$|D^r \sigma (x , \xi )| \leq C_{\sigma , r}, \esp \esp |D^r \sigma^* (x,\xi)| \leq C_{\sigma^* , r} ,$$ for every $r \geq 0$ and every $\xi \in \widehat{\Z}_p$. Hence by Lemma \ref{p-adicl1Hsnorms} we obtain $$||\widehat{J_r \sigma (\cdot , \xi)}||_{\ell^1 (\widehat{\Z}_p)} \leq C_{\sigma ,r}, \esp \esp ||\widehat{J_r \sigma^* (\cdot , \xi)}||_{\ell^1 (\widehat{\Z}_p)} \leq C_{\sigma^* ,r},$$ for every $\xi \in \widehat{\Z}_p$ and every $r \geq 0$. Thus $T_\sigma \in \mathcal{S}_r (\Z_p)$ for every $r \geq 0$. Conversely, if $T \in \mathcal{S}_r (\Z_p)$ for every $r \geq 0$ then $$|D^r \sigma (x,\xi)| \lesssim ||\widehat{J_r \sigma_T (\cdot , \xi)}||_{\ell^1 (\widehat{\Z}_p)} \leq C_{T ,r},$$ for every $r \geq 0$ and every $\xi \in \widehat{\Z}_p$. Then $\sigma_T  \in \Tilde{S}^0_{0,0} (\Z_p \times \widehat{\Z}_p)$.
\end{proof}
\begin{coro}
The class $Op(\Tilde{S}^0_{0,0} (\Z_p \times \widehat{\Z_p}))$ is inverse closed.
\end{coro}
We can prove a similar result for the operator classes $Op(\Tilde{S}^m_{0,0} (\Z_p \times \widehat{\Z_p}))$, $m \neq 0$, if we properly modify the definition of Schur classes.
\begin{defi}\normalfont
\esp
Let $M$ be an infinite matrix indexed by $\widehat{\Z}_p$. We say that $M$ is in the \emph{Schur class} $\mathcal{S}_r^m (\widehat{\Z}_p)$, $r \geq 0$ and $m \in \R$, if $$||M||_{\mathcal{S}_r^m (\widehat{\Z}_p)}:= \max \big\{ \sup_{\xi \in \widehat{\Z}_p} \langle \xi \rangle^{-m}\sum_{\eta \in \widehat{\Z}_p} |M_{\eta \xi}| \langle \eta - \xi \rangle^r , \sup_{\eta \in \widehat{\Z}_p} \langle \eta \rangle^{-m} \sum_{\xi \in \widehat{\Z}_p} |M_{\eta \xi}| \langle \eta - \xi \rangle^r  \big\}< \infty.$$We will denote by $\mathcal{S}_r^m (\Z_p)$ the class of linear operators $T$ such that its associated infinite matrix $M_T$ belongs to the Schur class $\mathcal{S}_r^m (\widehat{\Z}_p)$. That is: $$\mathcal{S}_r(\Z_p):= \{T:\hb^\infty \subset D(T) \subseteq L^2 (\Z_p) \to L^2 (\Z_p) \esp : \esp M_T \in \mathcal{S}_r^m (\widehat{\Z}_p) \}.$$We fix the notation $$||T||_{\mathcal{S}_r (\Z_p)} := ||M_T||_{\mathcal{S}(\widehat{\Z}_p)}.$$ 
\end{defi}
So, with the same arguments as before one can easily prove:
\begin{pro}
Let $T:\hb^\infty \subset D(T) \subseteq L^2 (\Z_p) \to L^2 (\Z_p)$ be a densely defined linear operator. Then $T \in Op(\Tilde{S}^m_{0,0} (\Z_p \times \widehat{\Z_p}))$ if and only if $T \in \bigcap_{r \geq 0} \mathcal{S}_r^m (\Z_p)$. That is: $$Op(\Tilde{S}^m_{0,0} (\Z_p \times \widehat{\Z}_p))= \bigcap_{r \geq 0} \mathcal{S}_r^m(\Z_p).$$
\end{pro}
\begin{coro}\label{coroboundedinverse}
Let $T_\sigma$ be a pseudo-differential operator with symbol $\sigma \in \Tilde{S}^m_{0,0} (\Z_p \times \widehat{\Z}_p)$. Then if $T_\sigma \in \mathcal{L}(H^{s+m}(\Z_p), H^{s} (\Z_p))$ is boundedly invertible, that is, there exists a linear operator $T_\sigma^{-1} \in \mathcal{L}(H^s(\Z_p), H^{s+m}(\Z_p))$ such that $T_\sigma^{-1}T_\sigma = I_{H^{s+m}(\Z_p)} , \esp T_\sigma T_\sigma^{-1} = I_{H^{s}(\Z_p)}$, its inverse $T_\sigma^{-1}$ is a pseudo-differential operator with symbol $\sigma^{-1}$ in the H{\"o}rmander class $\Tilde{S}^{-m}_{0,0} (\Z_p \times \widehat{\Z}_p)$.
\end{coro}
\begin{proof}
If $T_\sigma$ is invertible then $T_\sigma J_{-m} , J_{-m} T_\sigma \in \Tilde{S}^{0}_{0,0} (\Z_p \times \widehat{\Z}_p)$ are invertible as well with inverse $J_{m} T_\sigma^{-1}, T_\sigma^{-1} J_{m} \in \Tilde{S}^{0}_{0,0} (\Z_p \times \widehat{\Z}_p)$ respectively. This conclude the proof.
\end{proof}
We finally arrived to the main point of this section. We proved a relation between the Schur infinite matrix algebra and the H{\"o}rmander classes that we defined in the present work. The arguments used along this section are also valid in the toroidal case. That is, we can prove the equality $$Op(S^m_{0,0} (\T^d \times \Z^d))= \bigcap_{r \geq 0} \mathcal{S}_r^m(\T^d),$$ where the classes $\mathcal{S}_r^m(\Z^d)$ and $\mathcal{S}_r^m(\T^d)$ are analogously defined for infinite matrices indexed by $\Z^d$ and densely defined linear operators $T:L^2 (\T^d) \to L^2 (\T^d)$. The definition of the toroidal H{\"o}rmander classes is given in \cite{ruzhansky1}. This is our argument for asserting that the proper definition of H{\"o}rmander classes should include the estimates on the Vladimirov operator applied to the symbol: in that way our pseudo-differential operators have associated matrices with desirable properties in order to assure that they belong to the well known infinite matrix classes that may be found in the literature such as the Jaffard classes \cite{Jaffard}. However, as we mentioned before, in all this analysis we have not included the special properties of the non-archimidean absolute values on $\Z_p$, but we will do that in the next section.

\section{\textbf{Symbolic Calculus} II}
Recall the notation in Definition \ref{hormanderclasses2}. In \cite{pseudosvinlekin} L. Saloff-Coste proposed the following definition of symbol classes: Let $m \in \R$ and $0 \leq \delta \leq  \rho \leq 1$ be real numbers. A continuous function $\sigma : \Z_p \times  \widehat{\Z_p} \to \C$ belongs to the symbol class $\Check{S}^m_{\rho , \delta} (\Z_p \times  \widehat{ \Z}_p)$ if the following estimate holds: $$|\triangleplus_y^x \triangleplus_\eta^\xi \sigma (x , \xi)| \leq C_{m , \alpha , \beta, \rho , \delta } |y|_p^\beta |\eta|_p^\alpha \langle \xi \rangle^{m - \rho \alpha + \delta \beta},$$ for some constant $C_{ \rho , \delta } > 0$ and every $\alpha , \beta \in \N_0$. For linear operators with associated symbols in these classes L. Saloff-Costeau proved for $\delta = 0 , \rho = 1$, among other important properties, the following:
\begin{enumerate}
    \item[(i)] If $\sigma_1 \in \Check{S}^{m_1}_{\rho , \delta} (\Z_p \times  \widehat{ \Z}_p)$ and $\sigma_2 \in \Check{S}^{m_2}_{\rho , \delta} (\Z_p \times  \widehat{ \Z}_p)$ then the associated symbol of $T_{\sigma_1} T_{\sigma_2}$, let us call it $\sigma$, belongs to the symbol class $\Check{S}^{m_1 + m_2}_{\rho , \delta} (\Z_p \times  \widehat{ \Z}_p)$, and moreover $$\sigma - \sigma_1 \sigma_2 \in \bigcap_{m \in \R} \Check{S}^m_{\rho , \delta} (\Z_p \times  \widehat{ \Z}_p).$$
    \item[(ii)] If $\sigma \in \Check{S}^m_{\rho , \delta} (\Z_p \times  \widehat{ \Z}_p)$ the symbol $\sigma^*$ of the adjoint operator $T_\sigma^*$ belongs to the symbol class and moreover $$\sigma - \overline{\sigma} \in \bigcap_{m \in \R} \Check{S}^m_{\rho , \delta} (\Z_p \times  \widehat{ \Z}_p).$$
\end{enumerate}
The above gives one a symbolic calculus: a composition formula and an adjoint formula. This calculus is very special because for many purposes operators in the class $$\bigcap_{m \in \R} \Check{S}^m_{\rho , \delta} (\Z_p \times  \widehat{ \Z}_p) = \bigcap_{m \in \R} \Check{S}^m_{ 0,0} (\Z_p \times  \widehat{ \Z}_p) = \bigcap_{m \in \R} \Tilde{S}^m_{0 , 0} (\Z_p \times  \widehat{ \Z}_p),$$which is the analogue of the class of infinitely smoothing operators, are negligible. Indeed, properties of the calculus defined by L. Saloff-Coste are quite different to the well known archimedean setting, and in some sense they are much better. We will try to make the last statement precise at the same time that we show that our H{\"o}rmander classes $\Tilde{S}^m_{\rho , \delta} (\Z_p \times  \widehat{ \Z}_p)$ have the same good properties as $\Check{S}^m_{1, 0} (\Z_p \times  \widehat{ \Z}_p)$. 

First, we give our version of symbolic calculus. 

\begin{pro}
Let $0 <\rho \leq 1$. Let $T_{\sigma_1}$, $T_{\sigma_2}$ be pseudo-differential operators with symbols in the Hörmander classes $\sigma_1 \in \Tilde{S}^{m_1}_{\rho , 0} (\Z_p \times \widehat{\Z}_p)$,  $\sigma_2 \in \Tilde{S}^{m_2}_{ \rho , 0} (\Z_p \times \widehat{\Z}_p)$. Then :
\begin{enumerate}
    \item[(i)] $T_\sigma T_\tau = T_{\sigma \tau} + R$, where $R \in Op({\Tilde{S}^{-\infty}} (\Z_p \times \widehat{\Z}_p))$ and we have $\sigma_1 \sigma_2 \in \Tilde{S}^{m_1 + m_2}_{\rho, 0} (\Z_p \times \widehat{\Z}_p)$.
    \item[(ii)] If $\sigma \in \Tilde{S}^{m}_{\rho , 0} (\Z_p \times \widehat{\Z}_p)$ then $T_\sigma^t , T_\sigma^* \in Op(\Tilde{S}^{m}_{\rho , 0} (\Z_p \times \widehat{\Z}_p))$ and $\sigma^* - \overline{\sigma}, \sigma^t - \sigma \in {\Tilde{S}^{-\infty}} (\Z_p \times \widehat{\Z}_p).$.
\end{enumerate}
\end{pro}
\begin{proof}

\esp 

\begin{enumerate}
    \item[(i)] The symbol $\sigma (x, \xi)$ of the composition is given by \begin{align*}
        \sigma (x,\xi) &= \overline{ \chi_p (\xi x )} T_{\sigma_1} T_{\sigma_2} (\chi_p(\xi x)
    )  = \sum_{\eta \in \widehat{\Z}_p}  \sigma (x, \xi + \eta ) \widehat{\tau} (\eta,\xi) \chi_p (\eta x) \\ &= \sigma_1 (x,\xi) \sigma_2 (x, \xi) + \sum_{\eta \in \widehat{\Z}_p} \triangleplus_\eta^\xi \sigma_1 (x, \xi) \widehat{\sigma}_2 (\eta , \xi) \chi_p (\eta x).
    \end{align*} 
    Clearly $\sigma_1 \sigma_2 \in \Tilde{S}^{m}_{\rho , 0} (\Z_p \times \widehat{\Z}_p)$ because of Proposition \ref{normofproduct}. Also for the remainder we have \begin{align*}
        \Big| J^s \Big( \sum_{\eta , \nu \in \widehat{\Z}_p}  \triangleplus_\eta^\xi \widehat{\sigma}_1 (\nu , \xi) \widehat{\sigma}_2 (\eta , \xi) \chi_p (\eta x) \Big)  \Big| &= \Big|  \sum_{\eta , \nu \in \widehat{\Z}_p} \langle \eta + \nu \rangle^s \triangleplus_\eta^\xi \widehat{\sigma}_1 (\nu , \xi) \widehat{\sigma}_2 (\eta , \xi) \chi_p (\eta x)  \Big| \\ & \leq  \sum_{\eta , \nu \in \widehat{\Z}_p} \langle \eta + \nu \rangle^s |\triangleplus_\eta^\xi \widehat{\sigma}_1 (\nu , \xi)| \cdot |\widehat{\sigma}_2 (\eta , \xi) |.
    \end{align*}When $|\eta|_p \leq \langle \xi \rangle$ we get \begin{align*}
        \sum_{\eta , \nu \in \widehat{\Z}_p, |\eta|_p \leq \langle \xi \rangle} \langle \eta + \nu \rangle^s |\triangleplus_\eta^\xi \widehat{\sigma}_1 (\nu , \xi)| \cdot |\widehat{\sigma}_2 (\eta , \xi) | &\lesssim \sum_{\eta , \nu \in \widehat{\Z}_p, |\eta|_p \leq \langle \xi \rangle} \langle \eta + \nu \rangle^s |\eta|_p^\alpha \langle \xi \rangle^{m_1 + m_2 - \rho \alpha} \langle \nu \rangle^{-r} \langle \eta \rangle^{-r'},  
    \end{align*} for every $\alpha, r , r' \in \N_0$. Also we have $$|\widehat{\sigma}_2 (\eta , \xi)| \lesssim \langle \xi \rangle^{m_2} \langle \eta \rangle^{-r} \leq \langle \xi \rangle^{m_2 - \rho \alpha} \langle \eta \rangle^{r- \rho \alpha},$$ for every $r \in \N_0$ and $|\eta|_p > \langle \xi \rangle$. So, for $|\eta|_p > \langle \xi \rangle$  we get $$\sum_{\eta , \nu \in \widehat{\Z}_p, |\eta|_p > \langle \xi \rangle} \langle \eta + \nu \rangle^s |\triangleplus_\eta^\xi \widehat{\sigma}_1 (\nu , \xi)| \cdot |\widehat{\sigma}_2 (\eta , \xi) | \lesssim \langle \xi \rangle^{m_1 + m_2 - \rho \alpha}.$$In conclusion  $\Big| J^s \Big( \sum_{\eta , \nu \in \widehat{\Z}_p}  \triangleplus_\eta^\xi \widehat{\sigma}_1 (\nu , \xi) \widehat{\sigma}_2 (\eta , \xi) \chi_p (\eta x) \Big)  \Big| \lesssim \langle \xi \rangle^{m_1 + m_2 - \rho \alpha},$ for every $\alpha >0$ and any $s \in \R$. Thus $\sigma - \sigma_1 \sigma_2 \in {\Tilde{S}^{-\infty}} (\Z_p \times \widehat{\Z}_p) $.
    \item[(ii)] It is enough to prove $T_\sigma^t$ has its symbol in the class $\Tilde{S}^{m}_{\rho , 0} (\Z_p \times \widehat{\Z}_p).$ Recall that $T_\sigma^t$ is a linear operator $T_{a^t}$ associated to the amplitude $a^t (x,y,\xi) = \sigma (y,-\xi)$ whose symbol is
        
    \begin{align*}
    \sigma^t (x, \xi)  = \overline{\chi_p(\xi x)} T_{a^t} (\chi_p(\xi x)) &= \sum_{\eta \in \widehat{\Z}_p} \int_{\Z_p} \chi_p((\eta-\xi)(x-y)) a^t (x,y,\xi) dy\\&= \sum_{\eta \in \widehat{\Z}_p}  \widehat{\sigma} ( \eta - \xi , -\xi ) \chi_p(x(\eta-\xi))\\&= \sum_{\eta \in \widehat{\Z}_p}  \widehat{\sigma} (\eta, -(\xi + \eta)) \chi_p(x \eta).
\end{align*}From this we obtain \begin{align*}
    |D^\beta_x \triangleplus_\nu^\xi \sigma^t (x, \xi) | & \lesssim \sum_{\eta \in \widehat{\Z}_p} \langle \eta \rangle^\beta |\triangleplus_\nu^\xi \widehat{\sigma} (\eta, -(\xi + \eta))|,
\end{align*}so the following estimate $$|D^{\beta }_x \triangleplus_\nu^\xi \sigma (x,\xi)| \lesssim |\nu|^\alpha_p \langle \xi \rangle^{m_1 - \alpha},$$ for every $\beta \in \N_0$ implies $$|\triangleplus_\nu \widehat{\sigma} (\eta,-(\xi + \eta))| \lesssim |\nu|^\alpha \langle \xi + \eta \rangle^{m_1 -  \alpha } \langle \eta \rangle^{- r},$$for all $r \geq 0$. Hence $$|D^\beta_x \triangleplus_\nu^\xi \sigma^t (x, \xi) | \lesssim |\nu|^\alpha \langle \xi \rangle^{m_1 - \alpha},$$concluding the proof.
    \end{enumerate}
\end{proof}
There are some interesting consequences of the above calculus. For example: 
\begin{rem}
\esp
\begin{enumerate}
    \item[(i)] The adjoint operator $T_\sigma^*$ of $T_\sigma$ is $T_{\overline{\sigma}} + R$ where $R$ is a infinitely smoothing operator. The reason is because the symbol of the adjoint is given by \begin{align*}
        \sigma^* (x,\xi) &= \sum_{\eta \in \widehat{\Z}_p} \overline{\widehat{\sigma} (- \xi , \xi + \eta)} \chi_p (\eta x) \\ &= \overline{\sigma (x , \xi)} + \sum_{\eta 
        \in \widehat{\Z}_p} \triangleplus_{-\eta}^\xi \widehat{\sigma} (\eta , \xi) \chi_p (\eta x),
    \end{align*}where $$\sum_{\eta 
        \in \widehat{\Z}_p} \triangleplus_{-\eta}^\xi \widehat{\sigma} (\eta , \xi) \chi_p (\eta x) \in  S^{-\infty} (\Z_p \times \widehat{\Z}_p).$$The proof is analogue to the proof of \cite[IV.7]{pseudosvinlekin}.
    \item[(ii)] For every $n \in \N_0$ it holds $T_\sigma^n = T_{\sigma^n} + R$, $R \in S^{-\infty} (\Z_p \times \widehat{\Z}_p).$ As a consequence the formula $f(T_\sigma) = T_{f(\sigma)} + R$ holds for every complex analytic function. In particular this allows one to define and calculate the symbol (modulo an smoothing operator) of non integer powers of $T_\sigma$ trough the use of complex logarithm. Also for a positive real number $s$ we obtain $|T_\sigma|^s = (T_\sigma^* T_\sigma)^{s/2} = T_{|\sigma|^{s}} + R$.
    \item[(iii)] When $T_\sigma \in Op(\tilde{S}^m_{\rho,0} (\Z_p \times \widehat {\Z_p}))$ is boundedly invertible we know because of Corollary \ref{coroboundedinverse} that the corresponding inverse is in the class $\tilde{S}^{-m}_{0,0} (\Z_p \times \widehat {\Z_p})$, and we will prove something even better. From composition formula we know that $$T_\sigma T_\sigma^{-1} = T_\sigma T_{\sigma^{-1}} + R,$$so, one might expect for the symbol of the inverse $\sigma^{-1}$ to be $1/\sigma$ plus the symbol of a infinitely smoothing operator. This is actually true and moreover, there are some cases where the inverse of an operator in $Op(\tilde{S}^m_{\rho,0} (\Z_p \times \widehat {\Z_p}))$, or the inverse modulo a infinitely smoothing operator, does not belong to the class $Op(\tilde{S}^{-m}_{0,0} (\Z_p \times \widehat {\Z_p}))$, but we call still prove that it is in a certain H{\"o}rmander class and also that its associated symbol is $1/\sigma$ plus the symbol of an infinitely smoothing operator. We will discuss it in detail in the following section.
\end{enumerate}
\end{rem}

\section{\textbf{Fredholmness and hypoellipticity}}
An important class of pseudo-differential operators in the classical archimedean theory is the class of hypoelliptic operators. See for example the classical reference \cite{hypoellipticityHormander}. Its analogue in the present setting is defined as follows:
\begin{defi}\normalfont\label{hypoellipticity}
We say that a densely defined linear operator $$T: \hb^\infty \subset D(T) \to \mathfrak{D}(\Z_p),$$ is an hypoelliptic operator if the condition $Tf \in C^\infty (\Z_p)$ implies $f \in C^\infty (\Z_p)$.
\end{defi}
An important related concept is the definition of Fredholm operator:
\begin{defi}\normalfont
Let $E,F$ be Banach spaces. We say that $T \in \mathcal{L}(E,F)$ is a \emph{Fredholm operator} if there exists a $T^\bot \in \mathcal{L}(F,E)$ such that $$T T^\bot - I_F \in \mathfrak{K}(F), \esp \esp \text{and} \esp \esp T^\bot T - I_E \in \mathfrak{K}(E).$$
\end{defi}
In general it is a non-trivial problem to find necessary and sufficient conditions for the hypoellipticity of a pseudo-differential operator.For the case of Fourier multipliers it a necessary and sifficient condition is shown in \cite[Theorem 3.3]{2019arXiv190208237K}.

\begin{pro}
Let $T: \hb^\infty \subset D(T) \to \mathfrak{D}(\Z_p),$ be a densely defined Fourier multiplier with symbol $\sigma_T (\xi)$. Then $T$ is an hypoelliptic operator if and only if there exists a natural number $N \in \N_0$ such that $$C \langle \xi \rangle^{n} \leq |\sigma_T (\xi)|,$$for every $|\xi|_p \geq N$, some constant $C>0$ and some real number $n \in \R$. 
\end{pro}
For more general non-invariant operators in H{\"o}rmander classes we can give sufficient conditions in terms of the symbol for the hypoellipticity in H{\"o}rmander classes, but first we need to introduce some new definitions. The first one is the definition of $n$-hypoellipticity in the present setting. 
\begin{defi}\label{p-adicnhypoellipticity}\normalfont
Let $T_\sigma$ be a pseudo-differential operator with symbol $\sigma (x,\xi) \in \Tilde{S}^m_{0 , 0} (\Z_p \times \widehat{\Z}_p)$. Let $n \leq m$ be a real number. Then we say that $T_\sigma$ is $n$-\emph{hypoelliptic} if there exists an $N \in \N_0$ such that $\sigma(x,\xi)$ satisfies the following estimate $$\langle \xi \rangle^{n} \lesssim |\sigma (x,\xi)|,$$for $|\xi|_p \geq N$. When $n=m$ we simply say that $T_\sigma$ is an \emph{elliptic} operator.
\end{defi}
The definition of $n$-hypoellipticity is closely tied to the concept of Fredholmness for a pseudo-differential operator. To see this we now intend to prove that Definition \ref{p-adicnhypoellipticity} is equivalent to the following:
\begin{defi}\normalfont
Let $s,m,n$ be given real numbers. Let $T \in \mathcal{L}(H^{s+m} (\Z_p) , H^s (\Z_p))$ be a linear operator. We say that $T$ is $n$-pseudo invertible if there exists a linear operator $T^{\bot} \in \mathcal{L}(H^s (\Z_p) , H^{s+n}(\Z_p))$ such that the following equalities hold on $\hb^\infty := Span\{ \chi_p (\xi x) \}_{\xi \in \widehat{\Z}_p}$  $$T^\bot T = I + R_1, \esp \esp \text{and } \esp \esp T T^{\bot} = I + R_2,$$where $R_1, R_2$ extend to infinitely smoothing operators. The linear operator $T^{\bot}$ will be called the $n$-pseudo inverse of $T_\sigma$. 
\end{defi}

Now we want prove the following theorem, which is the analogue of \cite[Theorem 4.9.6]{ruzhansky1}.
\begin{teo}\label{hypoellipticifffred}
Let $0 < \rho \leq 1$. Let $T_\sigma$ be a pseudo-differential operator with symbol $\sigma (x,\xi) \in \Tilde{S}^m_{\rho , 0} (\Z_p \times \widehat{\Z}_p)$. Then $T_\sigma$ is an $n$-hypoelliptic operator if and only if it is $n$-pseudo-invertible with $n$-pseudo inverse $T^\bot \in \Tilde{S}^{-n}_{0 , 0} (\Z_p \times \widehat{\Z}_p)$. Moreover, when $m=n$ we have $T^\bot \in \Tilde{S}^{-m}_{\rho , 0} (\Z_p \times \widehat{\Z}_p)$.
\end{teo}
To prove the above we will need a lemma which is just a slight modification
of the weighted version of Wiener’s lemma in \cite{Grochenig2010wiener}.

\begin{lema}\label{inverseofhesymbol}
Let $\sigma$ be a symbol in the Hörmander class $\Tilde{S}^{m}_{\rho , 0} (\Z_p \times \widehat{\Z}_p)$ satisfying $$C_{\sigma, 1} \langle \xi \rangle^n \leq  | \sigma (x,\xi)|, \esp \esp |\xi|_p \geq p^{j_0}, \esp \esp n\leq m,$$for some $j_0 \in \N_0$. Then $1/\sigma \in S^{-n}_{0,0} (\Z_p \times \widehat{\Z_p})$.
\end{lema}

\begin{proof}
Let us assume $\sigma \in \Tilde{S}^{m}_{\rho , 0} (\Z_p \times \widehat{\Z}_p)$, $m \leq 0$. For the more general case $\sigma \in \Tilde{S}^{m}_{\rho , 0} (\Z_p \times \widehat{\Z}_p)$, $m \in \R$, one migth consider $\sigma' (x,\xi)= \sigma (x , \xi) \langle \xi \rangle^{-m},$ and obtain the desired conclusions. This lemma is in fact a corollary of Wiener's lemma \cite{Grochenig2010wiener, wienerschurlemma} or as well a consequence of the work of S. Jaffard \cite{Jaffard}. The idea is the following: by the weighted version of Wiener's Lemma if the functions $\{\sigma (x , \xi)\}_{\xi \in \widehat{\Z}_p}$ satisfy $$||\langle \eta \rangle^r \widehat{\sigma}(\cdot , \xi) ||_{\ell^1 (\widehat{\Z}_p)} \leq C_{\xi} , $$for every $\xi \in \widehat{\Z}_p$ and every $r \geq 0$, then also the following holds $$||\langle \eta \rangle^r \widehat{1/\sigma}(\cdot , \xi)||_{\ell^1 (\widehat{\Z}_p)} \leq C_\xi ',$$ for every $\xi \in \widehat{\Z}_p$ and every $r \geq 0$. We claim that actually, since $\sigma \in \Tilde{S}^{0}_{\rho , 0} (\Z_p \times \widehat{\Z}_p)$, $C_\xi =C$ does not depends on $\xi$ and hence $C_\xi ' = C'$ neither depends on $\xi$. Let us provide an sketch of the argument. 

Consider the function $h_\xi  (x) := \sigma(x, \xi)/||\sigma(\cdot , \xi)||_{L^\infty (\Z_p)} $. Clearly each $h_\xi$ is in the Banach algebra $\ell^1_r (\widehat{\Z}_p)$. Consider now the functions $$f_\xi (x) := 1- \frac{\sigma(x,\xi)}{||\sigma(\cdot, \xi)||_{L^{\infty} (\Z_p)}} = 1- h_\xi (x) .$$ Its $L^\infty$-norm is $$||f_\xi||_{L^\infty (\Z_p)} = \sup_{x \in \Z_p} \Big| 1- \frac{\sigma(x,\xi)}{||\sigma(\cdot, \xi)||_{L^{\infty} (\Z_p)}} \Big| \leq 1 -\frac{\inf_{x \in \Z_p}|\sigma(x,\xi)|}{||\sigma(\cdot, \xi)||_{L^{\infty} (\Z_p)}} <1,$$ and thus $h_\xi = 1 - f_\xi $ is invertible, and its inverse is given by $$h_\xi^{-1}=\sum_{k=0}^{\infty} f_\xi^k.$$Now we follow the proof in \cite[5.2.4]{Grochenig2010wiener}. First, we approximate $f_\xi (x) $ with a trigonometric polynomial $$q_\xi(x):= \sum_{|\eta|_p \leq p^{j_\xi}} \widehat{h}_\xi (\eta) \chi_p (\eta x), $$ in such a way that $||f_\xi - q_\xi ||_{\ell^1 (\widehat{\Z}_p)} < \varepsilon$, for some $\varepsilon>0$. Now, notice that, for a trigonometric polynomial $$q(x):= \sum_{|\xi|_p \leq p^j} \widehat{q}(\eta) \chi_p (\eta x) ,$$ we have the following estimate \begin{align*}
    ||q||_{\ell^1_r (\Z_p)} &:= \sum_{|\eta|_p \leq p^j} \langle \eta \rangle^r |\widehat{q}(\eta)|   \leq ||q||_{L^2 (\Z_p)} p^{rj} \big(  \sum_{|\eta|_p \leq p^j} 1 \big)^{1/2} \\ & \leq ||q||_{L^\infty (\Z_p)} p^{rj} \big(  \sum_{|\eta|_p \leq p^j} 1 \big)^{1/2}. 
    \end{align*}For the powers of $q$ we have the same estimate\begin{align*}
        ||q^n||_{\ell^1_r (\Z_p)} &\leq C_{j,r} ||q^n||_{L^\infty (\Z_p)} \\ & \leq C_{j,r} ||q||^n_{L^\infty (\Z_p)},
    \end{align*}where $$C_{j,r}= p^{rj}\big(  \sum_{|\eta|_p \leq p^j} 1 \big)^{1/2}   .$$The reason is that $q^n$ has the same order as $q$ because for the products of characters $$\chi_p (\xi_1 x) \cdot \chi_p (\xi_2 x) \cdot ... \cdot \chi_p (\xi_n x) = \chi_p ((\xi_1 + ... + \xi_n) x),$$ it holds $$|\xi_1 + ... + \xi_n|_p \leq \max \{ |\xi_1|_p , ..., |\xi_n|_p \}.$$  Writing $r_\xi :=f_\xi - q_\xi $, we obtain $$f^l_\xi = \sum_{k=0}^{l} {{l}\choose{k}} q^k_ \xi r^{l-k}_\xi,$$and then \begin{align*}
    ||f^l_\xi||_{\ell^1_r (\Z_p)} &\leq \sum_{k=0}^{l} {{l}\choose{k}} ||q^k_\xi||_{\ell^1_r (\Z_p)} ||r^{l-k}_{\xi}||_{\ell^1 (\widehat{\Z}_p)}\\&\lesssim \sum_{k=0}^{l} {{l}\choose{k}} \varepsilon^{l-k} ||q_\xi||^k_{L^{\infty} (\Z_p)} \\ &= (||q_\xi||_{L^{\infty} (\Z_p)} + \varepsilon)^k \\ &\leq (||f_\xi||_{L^{\infty} (\Z_p)} + 2 \varepsilon)^k  \\ &\leq(1- \delta_\xi + 2\varepsilon)^{k}.
\end{align*}Above we used the notation $$\delta_\xi := \frac{\inf_{x \in \Z_p} |\sigma (x, \xi)|}{\sup_{x \in \Z_p} |\sigma (x, \xi)|} .$$ Let us take a moment to show that our estimate is in fact independent of $\xi$. First, we have by hypothesis$$C_{\sigma , 1} \langle \xi \rangle^{n} \leq |\sigma (x,\xi)| \leq C_{\sigma , 2} \langle \xi \rangle^{m} \leq C_{\sigma , 2} ,$$where we may assume $C_{\sigma , 1} < C_{\sigma , 2}$. Thus$$\delta_\xi \geq \frac{C_{\sigma , 1}}{C_{\sigma , 2}} \langle \xi \rangle^{n-m}, \esp \esp |\xi|_p \geq p^{j_0},$$and hence, for appropriate $\varepsilon$, we get $$1 - \delta_\xi + 2 \varepsilon \leq 1 - \frac{C_{\sigma , 1}}{C_{\sigma , 2}} \langle \xi
\rangle^{n-m} + 2 \varepsilon<  1- \frac{C_{\sigma , 1}}{C_{\sigma , 2}} + 2 \varepsilon<1, \esp \esp |\xi|_p \geq p^{j_0} .$$ Also $\sigma \in \Tilde{S}^{m}_{\rho, 0} (\Z_p \times \widehat{\Z}_p) \subset \Tilde{S}^{0}_{0 , 0} (\Z_p \times \widehat{\Z}_p)$ implies the estimate $$|\widehat{\sigma} (\eta , \xi) | \leq C_{\sigma , k, 3 } \langle \eta \rangle^{-k}, \esp \esp \text{for all} \esp k \in \N_0,$$showing that we can choose some $j_1$ in such a way that $$||\sum_{|\eta|_p > p^{j_1}} \widehat{\sigma}(\eta , \xi) \chi_p (\eta x) ||_{\ell^1_r (\Z_p)} < \varepsilon,$$ for all $|\xi|_p \geq p^{\max\{j_0 , j_1\}}$. In conclusion we can estimate  \begin{align*}
    ||f_\xi^l||_{\ell^1_r (\Z _p)} & \lesssim \big( 1 - \frac{C_{\sigma , 1}}{C_{\sigma , 2}} + 2 \varepsilon \big)^{l} ,  
\end{align*}obtaining$$||1/h_\xi||_{\ell^1_r (\Z _p)} \leq C_{\sigma , r},$$and in conclusion $$||1/\sigma(\cdot , \xi)||_{\ell^1_r (\widehat{\Z}_p)} \leq C_{\sigma , r} \big( \sup_{x \in \Z_p} |\sigma(x,\xi)|\big)^{-1} \leq C_{\sigma , r} \langle \xi \rangle^{-n},$$ for every $|\xi|_p \geq p^{\max\{j_0 , j_1\}}$. All these prove $1/\sigma \in S^{-n}_{0,0} (\Z_p \times \widehat{\Z_p}).$ 
\end{proof}
\begin{rem}
The above proof works also in the case of the torus. In that setting we obtain:
\begin{lema}
Let $\sigma$ be a symbol in the Hörmander class $S^{m}_{0 , 0} (\T^d \times \Z^d)$ satisfying $$C\langle \xi \rangle ^n \leq  | \sigma (x,\xi)|, \esp \esp |\xi|_p \geq p^{j_0}, \esp \esp n \leq m.$$ Then $1/\sigma \in S^{-n}_{0,0} (\T^d \times \Z^d )$.
\end{lema}
\end{rem}

Now we can prove Theorem \ref{hypoellipticifffred}.
\begin{proof}[Proof of Theorem \ref{hypoellipticifffred}]
Using Lemma \ref{inverseofhesymbol} and composition formula is clear that if $T_\sigma$ is $n$-hypoellliptic then $T_{1/\sigma}$ defines a $n$-pseudo inverse. Conversely, if $T_\tau$ is a $n$-pseudo inverse of $T_\sigma$ with symbol $\tau \in S^{-n}_{0,0} (\Z_p \times \widehat{\Z_p})$ then by composition formula $\sigma \tau - 1 \in S^{-\infty} (\Z_p \times \widehat{\Z_p})$. Hence for large enough $|\xi|_p$ we have $$C \leq |\sigma (x , \xi) \tau (x , \xi) | \implies C \langle \xi \rangle^n\leq  \frac{C}{|\tau (x , \xi)|} \leq |\sigma (x,\xi)|.$$In the case when $m=n$ we can see that $T_\sigma \in S^m_{1 , 0} (\Z_p \times \widehat{\Z_p})$ is $m$-pseudo invertible with $m$-pseudo inverse $T^\bot$ if and only if $T_\sigma J_{-m} \in S^0_{\rho , 0} (\Z_p \times \widehat{\Z_p})$ is $0$-pseudo invertible with pseudo inverse $J_m T^\bot$. In this case we can use the fact that $Op(S^0_{1 , 0} (\Z_p \times \widehat{\Z_p}))$ is a $*$-subalgebra of $\mathcal{L}(L^2 (\Z_p))$ with the operator norm and \cite[Theorem A.1.3]{fredholmalg} to conclude $J_m T^\bot \in  Op(S^0_{1 , 0} (\Z_p \times \widehat{\Z_p}))$ and then $ T^\bot \in  Op(S^{-m}_{1 , 0} (\Z_p \times \widehat{\Z_p}))$.
\end{proof}

Next we show that for $T_\sigma \in Op(\tilde{S}^m_{\rho,0} (\Z_p \times \widehat {\Z_p}))$ Fredholmness and ellipticity are equivalent properties. With that purpose we state the characterization of compact operators in H{\"o}rmander classes. The proof of this statement is given in \cite{spectraltheoryp-adicpseudos}.  
\begin{lema}\label{compactpseudos}
Let $T_\sigma \in Op(S^0_{0 , 0} (\Z_p^d \times \widehat{\Z}_p^d))$ be a pseudo-differential operator. Then $T_\sigma$ extends to compact operator on $L^2(\Z_p^d)$ if and only if $$d_\sigma :=\limsup_{|\xi|_p \to \infty} || \sigma (\cdot , \xi)||_{L^\infty (\Z_p^d)} = 0.$$
\end{lema}

\begin{teo}
Let $T_\sigma \in Op(\tilde{S}^m_{\rho,0} (\Z_p \times \widehat {\Z_p}))$ be a pseudo-differential operator. Consider $T_\sigma$ as a bounded operator in $\mathcal{L}(H^m (\Z_p) , L^2 (\Z_p))$. Then $T_\sigma$ is a Fredholm operator if and only if it is elliptic.  
\end{teo}
\begin{proof}
Because of Theorem \ref{hypoellipticifffred} ellipticity implies Fredholmness. Conversely, if $T_\sigma$ is Fredholm, let us say $$T_\sigma T^\bot - I_{L^2 (\Z_p) } \in \mathfrak{K}(L^2 (\Z_p)) , \esp \esp T^\bot T_\sigma - I_{H^m (\Z_p)} \in \mathfrak{K}(H^m (\Z_p)),$$then $T^\bot= T_\beta \in Op(S^{-m}_{\rho , 0} (\Z_p \times \Z_p))$ because of \cite[Theorem A.1.3]{fredholmalg}. By composition formula $T_{\sigma \beta} - I $ is compact and Lemma \ref{compactpseudos} tell us that $$|\beta (x , \xi) \sigma (x , \xi) - 1| < 1/2, $$ for large $|\xi|_p$. Then $$|\sigma (x , \xi)| \gtrsim \frac{1}{|\beta (x , \xi)|} \gtrsim \langle \xi \rangle^m,$$ for large $|\xi|_p.$ This conclude the proof.
\end{proof}

As a corollary of our work on hypoellipticity we obtain some information about the possible closed domains for a $n$-hypoelliptic pseudo-differential operator. In what follows $T_{min}$ will denote the minimal operator of $T$ and $T_{max}$ the maximal operator on $L^2 (\Z_p)$. These are as usual the minimum and the maximal closed extensions of a densely defined linear operator. It is an 
\begin{pro}\label{equivalentnorms}
Let $T_\sigma \in Op(\Tilde{S}^{m}_{\rho, 0} (\Z_p \times \widehat{\Z}_p))$ be a $n$-hypoelliptic pseudo-differential operator, $ n \leq m$. Then there exist positive constants $C$ and $D$ such that $$C ||f||_{H^{s+n}(\Z_p) } \leq ||Tf||_{H^{s}(\Z_p)} + ||f||_{H^s(\Z_p) }\leq D ||f||_{H^{s+m}(\Z_p) }. $$
\end{pro}

\begin{proof}
First suppose that $$T_\beta T_\sigma = I + R,$$ where $T_\beta \in Op(\Tilde{S}^{-n}_{0 , 0} (\Z_p \times \widehat{\Z}_p)).$ Then for $f \in H^{s+m} (\Z_p) \subseteq H^{s+n} (\Z_p)$ \begin{align*}
||f||_{H^{s+n} (\Z_p)}  &= ||(T_\beta T_\sigma- R)f||_{H^{s+n} (\Z_p)} \\ &\leq ||T_\beta||_{\mathcal{L}(H^{s} (\Z_p) , H^{s+n} (\Z_p))} ||T_\sigma f||_{H^{s} (\Z_p)} +||R||_{\mathcal{L} (H^{s}(\Z_p) , H^{s+n} (\Z_p))}||f||_{H^{s} (\Z_p)}.
\end{align*}
Conversely, the inequality $$||Tf||_{H^{s} (\Z_p)} + ||f||_{H^{s} (\Z_p) }\leq D ||f||_{H^{s+m} (\Z_p)}$$ follows from the boundedness of $T_\sigma  \in Op(\Tilde{S}^{m}_{\rho , 0} (\Z_p \times \widehat{\Z}_p))$.
\end{proof}
\begin{pro}\label{min domain}
Let $T_\sigma \in Op(\Tilde{S}^{m}_{\rho , 0} (\Z_p \times \widehat{\Z}_p))$, $m>0$, be a $n$-hypoelliptic pseudo-differential operator, $0< n \leq m$. Then the domain of $(T_\sigma)_{min}$, considering $(T_\sigma)_{min}$ as closed operators acting on $H^s(\Z_p)$, lie between $H^{s+m}(\Z_p)$ and $H^{s+n}(\Z_p)$. In particular, when $n=m$ the exact domain of $(T_\sigma)_{min}$ is $H^{s+m}(\Z_p)$.
\end{pro}
\begin{proof}
Let $g \in H^{s+m}(\Z_p)$. Then by using the density of $\hb^\infty$ in $H^{s+m}(\Z_p)$, there exists a sequence $(\phi_k)_{k \in \N}$ in $\hb^\infty$ such that $\phi_k \to g$ in $H^{s+m}(\Z_p)$ and therefore in $H^{s}(\Z_p)$ as $k \to \infty$. By Proposition \ref{equivalentnorms}, $(\phi_k)_{k \in \N}$ and $(T_\sigma \phi_k)_{k \in \N}$ are Cauchy sequences in $H^{s}(\Z_p)$. Therefore $\phi_k \to g$ and $T_\sigma \phi_k \to f$ for some $f \in H^{s}(\Z_p)$ as $k \to \infty$. This implies that $g \in D((T_\sigma)_{min})$ and $(T_\sigma)_{min} g = f,$ hence $H^{s+m}(\Z_p) \subseteq D((T_\sigma)_{min})$. Now assume that $g \in D((T_\sigma)_{min})$. Then there exists a sequence $(\phi_k)_{k \in \N}$ in $\hb^\infty$ such that $\phi_k \to g$ in $H^{s}(\Z_p)$ and $T_\sigma \phi_k \to f$, for
some $f \in H^{s}(\Z_p)$. So, by Proposition \ref{equivalentnorms}, $(\phi_k)_{k \in \N}$ is a Cauchy sequence in $H^{s+n}(\Z_p)$. Since $H^{s+n}(\Z_p)$ is complete, there exists $h \in H^{s+n}(\Z_p)$ such that $\phi_k \to h$ in $H^{s+n}(\Z_p)$. This implies $\phi_k \to h$ in $H^{s}(\Z_p)$ which implies that $h = g \in H^{s+n}(\Z_p).$
\end{proof}
\begin{coro}\label{exactdomainmin}
If $n=m$ then $Dom((T_\sigma)_{min})=H^{s+m}(\Z_p)$.
\end{coro}
\begin{pro}\label{min-max domain}
Let $T_\sigma \in Op(\Tilde{S}^{m}_{\rho , 0} (\Z_p \times \widehat{\Z}_p))$ be a $n$-hypoelliptic pseudo-differential operator, $0<n \leq m$. Then the domains of $(T_\sigma)_{min}$ and  $(T_\sigma)_{max}$, considering $(T_\sigma)_{min} , (T_\sigma)_{max}$ as closed operators acting on $H^{s}(\Z_p)$, lie between $H^{s+m}(\Z_p)$ and $H^{s+n}(\Z_p)$.
\end{pro}

\begin{proof}
Since $(T_\sigma)_{max}$ is a closed extension of $(T_\sigma)_{min}$, by Proposition \ref{min domain} it is enough to show that $D ((T_\sigma)_{max}) \subseteq H^{s+n}(\Z_p) $. Let $g \in D ((T_\sigma)_{max})$. Since $T_\sigma$ is $n$-hypoelliptic there exists $T^\bot \in \mathcal{L} (H^{s}(\Z_p) ,H^{s+n}(\Z_p))$, $0 < n \leq m$, such that $$g = (T^\bot T_\sigma - R) g.$$ Since $Tg = T_{max} g \in H^{s}(\Z_p)$ it follows that $g \in H^{s+n}(\Z_p)$, which completes the proof
\end{proof}
\begin{coro}\label{maxmindomain}
Let $T_\sigma \in Op(\Tilde{S}^{m}_{\rho , 0} (\Z_p \times \widehat{\Z}_p))$ be a elliptic pseudo-differential operator. Then $(T_\sigma)_{min} = (T_\sigma)_{max}$, considering $(T_\sigma)_{min} , (T_\sigma)_{max}$ as closed operators acting on $H^{s}(\Z_p)$.
\end{coro}
\section{\textbf{Final remarks}}
An immediate application of the results developed in the past section is the $n$-hypoelliptic regularity for solutions of pseudo-differential equations.
\begin{coro}
Let $T_\sigma \in Op(S^{m}_{1 ,0} (\Z_p \times \widehat{\Z_p}))$ be a $n$-hypoelliptic pseudo-differential operator. Then if $T_\sigma u = f$, where $f \in H^{s} (\Z_p)$, we have $u \in H^{s + n} (\Z_p)$.
\end{coro}
More applications of the H{\"o}rmander classes and the symbolic calculus developed in this work are presented in \cite{spectraltheoryp-adicpseudos}. There we treat mostly spectral properties of pseudo-differential operators and its relation with the associated symbol. We encourage the interested reader to read that work because, in addition to being a continuation of this document, it also allows one to see more clearly the differences and advantages of the H{\"o}rmander classes and the symbolic calculus on $\Z_p$. For example the simplicity of the composition formula in this setting allows one to provide necessary and sufficient conditions for belonging to Schatten-Von Neumann classes. Also in \cite{spectraltheoryp-adicpseudos} the following version of the Weyl law is proved. We will use the special case when $d=1.$
\begin{teo}\label{weyllaw}
Let $T_\sigma$ be a pseudo-differential operator with symbol $\sigma \in \Tilde{S}^m_{\rho,0} (\Z_p^d \times \widehat{\Z}_p^d)$ acting on $L^2 (\Z_p^d)$. For $N\in \N_0$ let us define $$A_N(\sigma):=  \overline{  \bigcup_{||\xi||\geq p^N}   \{ \sigma (x,\xi) \esp : \esp x \in \Z_p^d\}}.$$ Assume that $\sigma$ is $n$-hypoelliptic of order $0  < n \leq m,$ let us say $$|\sigma (x,\xi)| \gtrsim \langle \xi \rangle^{n}, \esp \esp || \xi ||_p \geq p^{N_0},$$ and that there exists a curve $\gamma: [0, \infty) \to \C$ with the following properties: \begin{enumerate}
    \item[(i)] $\lim_{s \to \infty } |\gamma(s)| = \infty.$ 
    \item[(ii)]For large $s \in [0 , \infty)$ it holds: $$\sup_{\xi \in \widehat{\Z}_p^d} \frac{||\widehat{\sigma}(\cdot , \xi)||_{\ell^1 (\widehat{\Z}_p^d)} - \big| \int_{\Z_p^d} \sigma (x , \xi) dx \big|}{\big| \int_{\Z_p^d} \sigma (x , \xi) dx  - \gamma (s) \big|} < 1,$$and $$\sup_{\xi \in \widehat{\Z}_p^d} \frac{||\widehat{\sigma}^* (\cdot , \xi)||_{\ell^1 (\widehat{\Z}_p^d)} - \big| \int_{\Z_p^d} \sigma^* (x , \xi) dx \big|}{\big| \int_{\Z_p^d} \sigma^* (x , \xi) dx  - \gamma (s) \big|} < 1.$$ 
    \item[(iii)] For large $s \in [0,\infty)$ and $N_0 \in \N_0$ as above, and some  $\alpha \in [0,1)$, it holds $$|\gamma(s)|^{1- \alpha} \leq C_\gamma dist(\gamma(s),A_{N_0}(\sigma)).$$
\end{enumerate} Then we have the estimate $$N(t) \lesssim t^{\frac{d + \alpha(4n - d) )}{n}} \leq t^{\frac{d}{n} + 4 \alpha},$$where $$N(t):= \sum_{|\lambda_k (T_\sigma)| \leq t}1 .$$In consequence $$\esp k^{\frac{n}{d+ \alpha(4n - d)}} =O(|\lambda_k(T_\sigma)|) .$$ 
\end{teo}
As a corollary of the above version of the Weyl law, we can prove the smoothness of solutions to the heat equation.
\begin{teo}
Let $T_\sigma \in Op(S^m_{1 , 0} (\Z_p \times \widehat{\Z}_p))$, $m>0$, be a pseudo-differential operator that satisfies the hypothesis in Theorem \ref{weyllaw}. Assume that $T_\sigma$ possess a basis of eigenfunctions of $L^2(\Z_p)$ and the real part of its associated eigenvalues $\mathfrak{Re}(\lambda_j)$ grow as $|\lambda_j|$. Then the solutions to the equation $$(T_\sigma + \frac{\partial}{\partial t} ) f =0 , \esp\esp f(0,x) \in L^2 (\Z_p),$$ belong to $C^\infty (\Z_p)$ for all $t>0$.
\end{teo}
\begin{proof}
Let $\{e_j \}_{j \in \N}$ be the associated normalized basis of eigenfunctions of $T_\sigma$. Then the solution to the equation may be written as $$f(t,x) = \sum_{j \in \N} e^{-t \lambda_j } f_j e_j,$$where $f_j$ denotes the $j$-th component of $f(0,x)$ with respect to the basis $\{e_j \}_{j \in \N}$. Since $$T_\sigma^k f(t,x) = \sum_{j \in \N} \lambda_j^k e^{-t \lambda_j } f_j e_j,$$ is in $L^2 (\Z_p)$ for every $k \in \N_0$ and every $t >0$, we conclude using Proposition \ref{maxmindomain} that $$f(t,x) \in \bigcap_{k \in \N} D(T_\sigma^k) \subset \bigcap_{k \in \N} H^{kn} (\Z_p) = C^\infty (\Z_p),$$concluding the proof.
\end{proof}
\begin{coro}
Let $0<s_1<...< s_n$ be given positive numbers, and consider the pseudo-differential equation \begin{align}
    \sum_{i=1}^n a_i (x) D^{s_i} f + \frac{\partial f}{\partial t}= T_\sigma f +  \frac{\partial}{\partial t}  f=0 , \esp \esp f(0,x) = f_0 \in L^2(\Z_p),
\end{align}where the $a_i'$s, $1 \leq i \leq n$, are positive real valued functions on $\Z_p$ bounded from below such that at least one of them don't have zero mean, and $$\sigma (x,\xi) = \sum_{i=1}^n a_i (x)\big( |\xi|_p^{s_i} + (1- p^{-1}) \frac{p^{- {s_i}}}{1 - p^{-({s_i} + 1)}} (1 - \delta_{\xi,o})\big).$$ Then the solution $f(t,x)$ in the time $t$ of the heat equation belong to $C^\infty (\Z_p)$ for all $t>0$.
\end{coro}

\begin{coro}
Let $V \in L^\infty (\Z_p)$ be a real valued function and take $s>1$. Then the solutions to the equation $$(D^s + V(x) + \frac{\partial}{\partial t} ) f =0 , \esp\esp f(0,x) \in L^2 (\Z_p),$$ belong to $C^\infty (\Z_p)$ for all $t>0$. 
\end{coro}

\section*{\textbf{Acknowledgments}}
The author thanks Professor Michael Ruzhansky for his help during the development of this work.

\nocite{*}
\bibliographystyle{acm}
\bibliography{main}
\Addresses

\end{document}